\documentclass[12pt, reqno]{amsart}
\usepackage{amsmath, amstext, amsthm, amsbsy, amssymb, amscd}
\usepackage{latexsym}

\newtheorem{theorem}{Theorem}[section]

\newtheorem{lemma}[theorem]{Lemma}
\newtheorem{proposition}[theorem]{Proposition}
\newtheorem{corollary}[theorem]{Corollary}
\theoremstyle{definition}
\newtheorem{definition}[theorem]{Definition}

\theoremstyle{remark}
\newtheorem{remark}[theorem]{Remark}

\newtheorem{notation}[theorem]{Notation}

\numberwithin{equation}{section}

\newcommand{\As}{\mathcal A^\sharp_{(\omega, \beta \omega)}}
\newcommand{\C}{\mathbb C}
\newcommand{\ch}{\text{\rm ch}}
\newcommand{\D}{\mathcal D }

\newcommand{\Ext}{{\rm Ext}}

\newcommand{\Hom}{{\rm Hom}}

\newcommand{\Q}{\mathbb Q }
\newcommand{\R}{\mathbb R }
\newcommand{\rk}{\text{\rm rk}}
\newcommand{\shH}{\mathcal H}

\newcommand{\w}{\tilde}
\newcommand{\W}{\widetilde}

\newcommand{\Z}{\mathbb Z}

\begin{document}

\title[Mini-walls for Bridgeland stability conditions]
{Mini-walls for Bridgeland stability conditions on the derived category
of sheaves over surfaces}

\author[Jason Lo]{Jason Lo}
\address{Department of Mathematics, University of Missouri,
Columbia, MO 65211, USA} \email{locc@missouri.edu}

\author[Zhenbo Qin]{Zhenbo Qin$^\dagger$}
\address{Department of Mathematics, University of Missouri,
Columbia, MO 65211, USA} \email{qinz@missouri.edu}
\thanks{${}^\dagger$Partially supported by an NSF grant}

\keywords{Walls, Bridgeland stability, polynomial stability, derived category.}
\subjclass{Primary 14D20; Secondary: 14F05, 14J60}

\begin{abstract}
 For the derived category of bounded complexes of sheaves on a smooth projective
surface, Bridgeland \cite{Bri2} and Arcara-Bertram \cite{ABL} constructed Bridgeland stability
conditions $(Z_m, \mathcal P_m)$ parametrized by $m \in (0, +\infty)$.
In this paper, we show that the set of mini-walls in $(0, +\infty)$ of
a fixed numerical type is locally finite. In addition, we strengthen a result of
Bayer \cite{Bay} by proving that the moduli of polynomial Bridgeland semistable objects of
a fixed numerical type coincides with the moduli of $(Z_m, \mathcal P_m)$-semistable objects
whenever $m$ is larger than a universal constant depending only on the numerical type.
We further identify the moduli of polynomial Bridgeland semistable objects with the
Gieseker/Simpson moduli spaces and the Uhlenbeck compactification spaces.
\end{abstract}

\maketitle
\section{\bf Introduction}
\label{sect_intro}

Since the appearance of Bridgeland's seminal work \cite{Bri1}, there have been intensive
investigations of Bridgeland stability conditions on triangulated categories,
which can be viewed as a mathematical approach to understand Douglas' work \cite{Dou}
on $\Pi$-stability for D-branes in string theory. Bridgeland stability conditions for
smooth projective curves were classified by Macri \cite{Mac} and Okada \cite{Oka}.
Bridgeland stability conditions on smooth projective surfaces were constructed by
Bridgeland \cite{Bri2} and Arcara-Bertram \cite{ABL}, and the topology of the stability
manifolds for generic K3 categories was obtained by Huybrechts, Macri and Stellari \cite{HMS}.
Toda \cite{Tod2} studied Bridgeland stability conditions for Calabi-Yau fibrations.
A gluing procedure for Bridgeland stability conditions was found by Collins and
Polishchuk \cite{CP}. In another direction,
Bayer \cite{Bay} (see also Toda \cite{Tod3}) defined polynomial Bridgeland
stability for normal projective varieties of any dimension. The polynomial Bridgeland
stability may be viewed as the large volume limit of the Bridgeland stability.
The moduli stacks of Bridgeland semistable objects were investigated
in \cite{Ina, Lie, Tod1}, while the moduli stacks of {\it polynomial}
Bridgeland semistable objects were investigated in \cite{Lo1, Lo2, Lo3, LQ}.

The concepts of walls and chambers for Gieseker stability were introduced in \cite{Qin}
and played an important role in understanding Donaldson polynomial invariants of certain surfaces.
Walls and chambers in the space of Bridgeland stability conditions are closely related
to the wall-crossing phenomena discussed by Kontsevich and Soibelman \cite{KS}.
Let $X$ be a smooth projective surface, and let $\mathcal D^b(X)$ be the derived category of
bounded complexes of coherent sheaves on $X$. When $X$ is a $K3$ or abelian surface,
Bridgeland \cite{Bri2} proved that the set of walls in the space of Bridgeland stability conditions
on $\mathcal D^b(X)$ is locally finite. Whether the same conclusion holds for a general surface $X$
remains to be open.

In this paper, we analyze the set of mini-walls and mini-chambers in the space of
Bridgeland stability conditions. To state our results, we introduce some notations and definitions
(see Notation~\ref{ZmPm} and Definition~\ref{mini-walls} below for details).
The {\it numerical type} of an object $E \in \mathcal D^b(X)$ is defined to be
$\mathfrak t(E) = (\rk(E), c_1(E), c_2(E))$. Fix $\beta, \omega \in \text{\rm Num}(X)_\Q$
with $\omega$ being ample, and fix a numerical type $\mathfrak t = (r, c_1, c_2)$.
Bridgeland \cite{Bri2} and Arcara-Bertram \cite{ABL} constructed Bridgeland stability
conditions $(Z_m, \mathcal P_m)$ parametrized by $m \in (0, +\infty)$.
Regard $(0, +\infty)$ as a subset in the space of Bridgeland stability conditions.
Then walls and chambers in $(0, +\infty)$ are referred to as {\it mini-walls} and
{\it mini-chambers} of type $(\mathfrak t, \beta, \omega)$.

\begin{theorem}       \label{Intro-thm1}
Let $\beta, \omega \in \text{\rm Num}(X)_\Q$ with $\omega$ being ample,
and let $\mathfrak t = (r, c_1, c_2)$.
\begin{enumerate}
\item[(i)] The set of mini-walls of type $(\mathfrak t, \beta, \omega)$
in $(0, +\infty)$ is locally finite.

\item[(ii)] There exists a positive number $\W M$, depending only on
$\mathfrak t, \omega$ and $\beta$, such that there is no
mini-wall of type $(\mathfrak t, \beta, \omega)$ in $[\W M, +\infty)$.
\end{enumerate}
\end{theorem}

Theorem~\ref{Intro-thm1} has been observed in the special case considered in
Sect.~4 of \cite{ABL}. Moreover, Theorem~\ref{Intro-thm1}~(ii) strengths
the Proposition~4.1 in \cite{Bay} (see Lemma~\ref{4.1} below).
In fact, we prove in Theorem~\ref{mggM} that whenever $m \ge \W M$,
an object $E \in \mathcal D^b(X)$ with $\mathfrak t(E) = \mathfrak t$
is $(Z_m, \mathcal P_m)$-semistable if and only if $E$ is
$(Z_\Omega, \mathcal P_\Omega)$-semistable. Here $\Omega = (\omega, \rho, p, U)$ is
the stability data from Subsect.~\ref{subsect_Large}, and $(Z_\Omega, \mathcal P_\Omega)$
denotes the polynomial Bridgeland stability constructed in \cite{Bay}.

The main idea in proving Theorem~\ref{Intro-thm1}~(i) is to find an upper bound
for $\rk(A)$ if $A$ defines a mini-wall of type $(\mathfrak t, \beta, \omega)$
and if the mini-wall is contained in an interval $I = [a, +\infty)$.
This upper bound is universal in the sense that
it depends only on $I$ and $(\mathfrak t, \beta, \omega)$. Combining this idea
with an expanded version of the proof of the Proposition~4.1 in \cite{Bay} also
leads to the proof of Theorem~\ref{Intro-thm1}~(ii).

Next, we classify all the polynomial Bridgeland semistable objects in terms of
Gieseker/Simpson semistable sheaves. Let $\overline{\mathfrak M}_\Omega(\mathfrak t)$
be the set of all $(Z_\Omega, \mathcal P_\Omega)$-semistable objects
$E \in \mathcal P_\Omega((0, 1])$ with $\mathfrak t(E) = \mathfrak t$.
Let $\overline{\mathfrak M}_{\omega}(\mathfrak t)$ be
the moduli space of sheaves $E \in \text{\rm Coh}(X)$ which are
Simpson-semistable with respect to $\omega$ and satisfy $\mathfrak t(E) = \mathfrak t$.
 For $r > 0$, define ${\mathfrak M}_{\omega}(\mathfrak t)$ be
the moduli space of locally free sheaves $E$ which are
$\mu_\omega$-stable and satisfy $\mathfrak t(E) = \mathfrak t$,
and define $\overline{\mathfrak U}_{\omega}(\mathfrak t)$ to be the Uhlenbeck
compactification space associated to $\omega$ and $\mathfrak t$.
The case when $r = 0$ is covered by Lemma~\ref{00n} and Lemma~\ref{0c1c2}. For $r \ne 0$,
we have the following.

\begin{theorem}     \label{Intro-thm2}
Let $\Omega = (\omega, \rho, p, U)$ be from Subsect.~\ref{subsect_Large}.
Fix a numerical type $\mathfrak t = (r, c_1, c_2)$.
Let $\w {\mathfrak t} = (-r, c_1, c_1^2-c_2)$.
Assume that $\omega$ lies in a chamber of type $\mathfrak t$.
\begin{enumerate}
\item[(i)] If $r > 0$, then
$\overline{\mathfrak M}_\Omega(\mathfrak t) \cong
 \overline{\mathfrak M}_{\omega}(\mathfrak t)$.

\item[(ii)] If $r < 0$ and $c_1 \omega/r < \beta \omega$,
then $\overline{\mathfrak M}_\Omega(\mathfrak t) \cong
\overline{\mathfrak M}_{\omega}(\w {\mathfrak t})$.

\item[(iii)] If $r < 0$ and $c_1 \omega/r = \beta \omega$,
then $\overline{\mathfrak M}_\Omega(\mathfrak t) \cong
\overline{\mathfrak U}_{\omega}(\w {\mathfrak t})$.
\end{enumerate}
\end{theorem}

We therefore have a complete description of the moduli spaces of
$(Z_\Omega, \mathcal P_\Omega)$-semistable objects on every smooth projective surface.
In view of Theorem~\ref{mggM}, we obtain a complete description of the moduli spaces of
semistable objects with respect to certain Bridgeland stabilities on
a smooth projective surface. We remark that similar results in the context of
Bridgeland stability have been observed and studied by Kawatani \cite{Kaw},
Ohkawa \cite{Ohk} and Toda \cite{Tod1}. Similar results in the context of
{\it polynomial} Bridgeland stability have also appeared in Sect.~5 of \cite{LQ}
which only considered objects $E \in \mathcal A^p$ for those stability data
$\Omega = (\omega, \rho, p, U)$ such that $\rho = (\rho_0, \rho_1, \rho_2)$ satisfies
$\phi(\rho_0) \ne \phi(-\rho_2)$. However, in our present situation,
we have $\phi(\rho_0) = \phi(-\rho_2)$ since $\rho_0 = -1$ and $\rho_2 = 1/2$.

This paper is organized as follows. In Sect.~2, we recall the constructions of
Bridgeland, Arcara-Bertram and Bayer. Theorem~\ref{Intro-thm1}~(i) and (ii) are proved
in Sect.~3 and Sect.~4 respectively. In Sect.~5, we verify Theorem~\ref{Intro-thm2}.

\bigskip\noindent
{\bf Conventions}: The $i$-th cohomology of a sheaf $E$ on
a variety $X$ is denoted by $H^i(X, E)$, and its usual dual sheaf
$\mathcal Hom(E, \mathcal O_X)$ is denoted by $E^*$.
The derived category of bounded complexes of coherent sheaves
on $X$ is denoted by $\mathcal D^b(X)$. The $i$-th cohomology
sheaf of an object $E \in \mathcal D^b(X)$ is denoted by $\shH^i(E)$,
and the derived dual of $E$ is denoted by $E^v =
\R \mathcal Hom(E, \mathcal O_X) \in \mathcal D^b(X)$.

\bigskip\noindent
{\bf Acknowledgment}:
The authors thank Professors Jun Li and Wei-Ping Li for valuable helps and
stimulating discussions.

\section{\bf Preliminaries}
\label{sect_Preliminaries}

\subsection{\bf Constructions of Bridgeland and Arcara-Bertram}
\label{subsect_Constructions}
\par
$\,$

Let $X$ be a smooth complex projective surface.

\begin{definition}   \label{TFA}
Let $\omega \in {\rm Num}(X)_\R$ be ample, and let $v \in \R$.
\begin{enumerate}
\item[(i)] Define $\mathcal T_{(\omega, v)}$ to be the full
subcategory of ${\rm Coh}(X)$ generated by
torsion sheaves and torsion free $\mu_\omega$-stable sheaves $A$
with $\mu_\omega(A) > v$.

\item[(ii)] Define $\mathcal F_{(\omega, v)}$ to be the full
subcategory of ${\rm Coh}(X)$ generated by torsion free
$\mu_\omega$-stable sheaves $A$ with $\mu_\omega(A) \le v$.

\item[(iii)] Define $\mathcal A^\sharp_{(\omega, v)}$ to be the abelian category
obtained from ${\rm Coh}(X)$ by tilting at the torsion pair $\big (\mathcal T_{(\omega, v)},
\mathcal F_{(\omega, v)} \big )$, i.e., $\mathcal A^\sharp_{(\omega, v)}$ consists of
all the objects $E \in \mathcal D^b(X)$ satisfying the conditions:
\begin{eqnarray}  \label{des-tilt}
\mathcal H^{-1}(E) \in \mathcal F_{(\omega, v)}, \quad
\mathcal H^0(E) \in \mathcal T_{(\omega, v)}, \quad
\mathcal H^i(E) = 0 \,\, \text{\rm for } i \ne -1, 0.
\end{eqnarray}
\end{enumerate}
\end{definition}

The following lemma will be used in Case 3 in the proof of
Lemma~\ref{semistable-bnd} below.

\begin{lemma} \label{GCQ}
Let $Q$ be a $0$-dimensional torsion sheaf, and $\mathcal C \in
\mathcal T_{(\omega, v)}$. If $\mathcal G$ sits in an exact
sequence $0 \to \mathcal G \to \mathcal C \to Q \to 0$ of coherent sheaves,
then $\mathcal G \in \mathcal T_{(\omega, v)}$.
\end{lemma}
\begin{proof}
Let ${\rm Tor}(\mathcal C)$ denote the torsion subsheaf of
$\mathcal C$. Let
$$
{\rm Tor}(\mathcal C) = \mathcal C_0 \subset \mathcal C_1
\subset \ldots \subset \mathcal C_n = \mathcal C
$$
be the usual HN-filtration of $\mathcal C$ with respect to $\mu_\omega$. Let $\mu_i =
\mu_\omega(\mathcal C_i/\mathcal C_{i-1})$. Then for $i =
1, \ldots, n$, the sheaf $\mathcal C_i/\mathcal C_{i-1}$
is torsion free and $\mu_\omega$-semistable.
Moreover, $\mu_1 > \ldots > \mu_n$. By the definition of
$\mathcal T_{(\omega, v)}$, $\mu_1 > \ldots > \mu_n > v$.
For $i = 0, 1, \ldots, n$, let $\mathcal G_i$ and
$\mathcal Q_i$ be the kernel and image of the induced map
$\mathcal C_i \to Q$ respectively. Then $\mathcal G_{i-1} =
\mathcal G_i \cap \mathcal C_{i-1}$. The injection
$0 \to \mathcal G_i/\mathcal G_{i-1} \to
\mathcal C_{i}/\mathcal C_{i-1}$ implies that
$\mathcal G_i/\mathcal G_{i-1}$ is torsion free.
Also, we have commutative diagram of sheaves
\begin{eqnarray*}
\begin{array}{cccccccccccc}
&&0&&0&&\ldots&&0&&0 \\
&&\downarrow&&\downarrow&&&&\downarrow&&\downarrow \\
{\rm Tor}(\mathcal G)&=&\mathcal G_0&\subset&\mathcal G_1
  &\subset&\ldots&\subset&\mathcal G_n&=&\mathcal G \\
&&\downarrow&&\downarrow&&&&\downarrow&&\downarrow \\
{\rm Tor}(\mathcal C)&=&\mathcal C_0&\subset&\mathcal C_1
  &\subset&\ldots&\subset&\mathcal C_n&=&\mathcal C \\
&&\downarrow&&\downarrow&&&&\downarrow&&\downarrow \\
&&\mathcal Q_0&\subset&\mathcal Q_1&\subset&\ldots&\subset
  &\mathcal Q_n&=&\mathcal Q \\
&&\downarrow&&\downarrow&&&&\downarrow&&\downarrow \\
&&0&&0&&\ldots&&0&&0
\end{array}
\end{eqnarray*}
from which we obtain two exact sequences for each $i =
1, \ldots, n$:
\begin{eqnarray*}
\begin{array}{cccccccccccc}
0&\to&\mathcal G_i/\mathcal G_{i-1}&\to&
  \mathcal C_i/\mathcal G_{i-1}&\to&\mathcal Q_i&\to&0 \\
&&&&\|&&&& \\
0&\to&\mathcal Q_{i-1}&\to&\mathcal C_i/\mathcal G_{i-1}
  &\to&\mathcal C_{i}/\mathcal C_{i-1}&\to&0. \\
\end{array}
\end{eqnarray*}
Since $\mathcal G_i/\mathcal G_{i-1}$ is torsion free and
$\mathcal Q_{i-1}$ is torsion, we get an exact sequence
\begin{eqnarray*}
\begin{array}{cccccccccccc}
0 \to \mathcal G_i/\mathcal G_{i-1} \to
\mathcal C_{i}/\mathcal C_{i-1} \to \W{\mathcal Q}_i \to 0
\end{array}
\end{eqnarray*}
where $\W{\mathcal Q}_i$ is a $0$-dimensional torsion sheaf.
Thus, $\mu_\omega(\mathcal G_i/\mathcal G_{i-1})
= \mu_\omega(\mathcal C_i/\mathcal C_{i-1}) = \mu_i > v$,
and $\mathcal G_i/\mathcal G_{i-1}$ is $\mu_\omega$-semistable.
Hence
$$
{\rm Tor}(\mathcal G) = \mathcal G_0 \subset \mathcal G_1
\subset \ldots \subset \mathcal G_n = \mathcal G
$$
is the usual HN-filtration of $\mathcal G$ with respect to
$\mu_\omega$, and $\mathcal G \in \mathcal T_{(\omega, v)}$.
\end{proof}

Let ${\bf u} \in {\mathcal N}(X) \otimes_\Z \C$. Define the charge $Z_{\bf u}$ on $\D^b(X)$ by
\begin{eqnarray}  \label{bri-charge}
Z_{\bf u}(E) = - \int_X {\bf u} \cdot \ch(E).
\end{eqnarray}
The following lemma is due to Bridgeland \cite{Bri2} and Arcara-Bertram \cite{ABL}.

\begin{lemma} \label{stb-D}
Let ${\bf u} = e^{-(\beta + i \, \omega)}$ where $\beta, \omega
\in {\rm Num}(X)_\R$ and $\omega$ is ample. Then
$\big ( Z_{\bf u}, \mathcal A^\sharp_{(\omega, \beta\omega)} \big )$
induces a Bridgeland stability condition
$(Z_{\bf u}, \mathcal P_{\bf u})$ on $\D^b(X)$.
\end{lemma}

\subsection{\bf Polynomial stability and large volume limits}
\label{subsect_Large}
\par
$\,$

Let $\Omega = (\omega, \rho, p, U)$ be the stability data defined
by the following:
\begin{enumerate}
\item[$\bullet$] $\omega \in \text{\rm Num}(X)_\R$ is ample,
\item[$\bullet$] $\rho = (\rho_0, \rho_1, \rho_2)$ with
                 $\rho_i = -(-i)^d/d!$,
\item[$\bullet$] $p: \{0, 1, 2\} \to \Z$ is the perversity function
                 $p(d) = -\lfloor d/2 \rfloor$,
\item[$\bullet$] $U = e^{-\beta}$ for some $\beta \in
                 \text{\rm Num}(X)_\R$.
\end{enumerate}
Let $Z_\Omega: K(\mathcal D^b(X))
= K(X) \to \C[m]$ be the central charge defined by
\begin{eqnarray}      \label{charge}
Z_\Omega(E)(m)
= \int_X \sum_{d=0}^2 \rho_d \omega^dm^d \cdot \ch(E) \cdot U
= - \int_X e^{-(\beta + i \, m \omega)} \cdot \ch(E).
\end{eqnarray}
By \cite{Bay}, $Z_\Omega(E)(m)$ induces
a polynomial stability condition $(Z_\Omega, \mathcal P_\Omega)$
on $\mathcal D^b(X)$.

\begin{lemma} \label{4.2}
(Lemma~4.2 in \cite{Bay}) We have $\mathcal P_\Omega((0, 1]) =
\mathcal A^\sharp_{(\omega, \beta\omega)}$. In fact,
if $E \in \mathcal P_\Omega((0, 1]) \subset \mathcal D^b(X)$ is
$(Z_\Omega, \mathcal P_\Omega)$-semistable, then $E$ is one of
the following:
\begin{enumerate}
\item[(i)] $E$ is a torsion sheaf;

\item[(ii)] $E$ is a torsion free $\mu_\omega$-semistable sheaf
with $\mu_\omega(E) > \beta \omega$.

\item[(iii)] $\mathcal H^{-1}(E)$ is a torsion free
$\mu_\omega$-semistable sheaf with $\mu_\omega(\mathcal H^{-1}(E))
\le \beta \omega$, $\mathcal H^0(E)$ is a $0$-dimensional torsion,
and all other cohomology sheaves of $E$ vanish.
\end{enumerate}
\end{lemma}

\begin{notation} \label{ZmPm}
Fix ${\bf u} = e^{-(\beta + i \, \omega)}$. Put ${\bf u}_m =
e^{-(\beta + i \, m \omega)}$ and
$$
Z_m(E) = Z_{{\bf u}_m}(E)
= - \int_X e^{-(\beta + i \, m \omega)} \cdot \ch(E)
= Z_\Omega(E)(m).
$$
Let $(Z_m, \mathcal P_m)$ denote the Bridgeland stability condition
on $\mathcal D^b(X)$ induced by $Z_m$.
\end{notation}

\begin{lemma} \label{4.1}
(Proposition~4.1 in \cite{Bay})
Let notations be as above. Assume further that $\omega \in
\text{\rm Num}(X)_\Q$. Then, an object $E \in \mathcal D^b(X)$ is
$(Z_\Omega, \mathcal P_\Omega)$-semistable
if and only if $E$ is $(Z_m, \mathcal P_m)$-semistable for $m \gg 0$.
Moreover, for an arbitrary object $E \in \mathcal D^b(X)$, the
HN-filtration of $E$ with respect to $(Z_\Omega, \mathcal P_\Omega)$
is identical to the HN-filtration of $E$ with respect to
$(Z_m, \mathcal P_m)$ for $m \gg 0$.
\end{lemma}

\begin{definition}   \label{num-type}
For $E \in \mathcal D^b(X)$, define its {\it numerical type}
$\mathfrak t(E)$ by
\begin{eqnarray}   \label{num-type.1}
\mathfrak t(E) = (\rk(E), c_1(E), c_2(E)).
\end{eqnarray}
\end{definition}

\begin{remark}  \label{compare-phase}
A straight-forward computation from (\ref{charge}) shows that
\begin{eqnarray} \label{ZEm}
Z_\Omega(E)(m) = \rk(E) \omega^2 \cdot {m^2 \over 2}
+ i (c_1(E) \cdot \omega - \rk(E) \, \beta \omega)m
+ c(E)
\end{eqnarray}
where
\begin{eqnarray} \label{cE}
c(E) = -\ch_2(E) + c_1(E) \cdot \beta - \rk(E) \cdot {\beta^2 \over 2}
\,\, \in \,\, \R.
\end{eqnarray}
It follows that if $E, B \in \mathcal P_\Omega((0, 1]) =
\mathcal A^\sharp_{(\omega, \beta\omega)}$ and $m > 0$, then
$\phi \big (Z_\Omega(E)(m) \big ) > \phi \big (Z_\Omega(B)(m) \big )$
is equivalent to
$\text{\rm Im} \big (\overline{Z_\Omega(E)(m)}
\cdot Z_\Omega(B)(m) \big ) < 0$, i.e.,
\begin{eqnarray}   \label{EB2}
& &{\omega^2 m^2 \over 2} \big ( \rk(E) \, c_1(B) \omega -
    \rk(B) \, c_1(E) \omega \big )   \nonumber  \\
&<&c(B) \, \big ( c_1(E) \omega - \rk(E) \, \beta \omega \big )
    - c(E) \, \big ( c_1(B) \omega - \rk(B) \, \beta \omega \big ).
\end{eqnarray}
\end{remark}

\subsection{\bf Moduli spaces, walls and chambers}
\label{subsect_Moduli}
\par
$\,$

\begin{definition}   \label{Gr}
Let notations be as above.
Fix a numerical type $\mathfrak t = (r, c_1, c_2)$.
\begin{enumerate}
\item[(i)] Let $E \in \mathcal P_\Omega((0, 1]) =
\mathcal A^\sharp_{(\omega, \beta\omega)}$ be
$(Z_\Omega, \mathcal P_\Omega)$-semistable, and let
$$
0 = E_0 \subset E_1 \subset \ldots \subset E_{n-1} \subset E_n = E
$$
be the Jordan-Holder filtration of $E$. Define
\begin{eqnarray}   \label{Gr.1}
\text{\rm Gr}(E) = \bigoplus_{i=1}^n E_i/E_{i-1}.
\end{eqnarray}
Two $(Z_\Omega, \mathcal P_\Omega)$-semistable objects $E_1, E_2 \in
\mathcal P_\Omega((0, 1])$ are defined to be {\it $S$-equivalent} if
$\text{\rm Gr}(E_1) \cong \text{\rm Gr}(E_2)$.
Define $\overline{\mathfrak M}_\Omega(\mathfrak t)$ to be the set of
all $(Z_\Omega, \mathcal P_\Omega)$-semistable objects $E \in
\mathcal P_\Omega((0, 1])$ with $\mathfrak t(E) = \mathfrak t$
modulo $S$-equivalence.

\item[(ii)] For $m > 0$, define the {\it $S$-equivalence with respect
to $(Z_m, \mathcal P_m)$} in a similar fashion as in (i), and define
$\overline{\mathfrak M}_{{\bf u}_m}(\mathfrak t)$
be the set of all $(Z_m, \mathcal P_m)$-semistable objects
$E \in \mathcal A^\sharp_{(\omega, \beta\omega)}$ with
$\mathfrak t(E) = \mathfrak t$ modulo $S$-equivalence.

\item[(iii)] Let $\overline{\mathfrak M}_{\omega}(\mathfrak t)$ be
the moduli space of sheaves $E \in \text{\rm Coh}(X)$ which are
Simpson-semistable with respect to $\omega$ and satisfy $\mathfrak t(E)
= \mathfrak t$.

\item[(iv)] For $r > 0$, define ${\mathfrak M}_{\omega}(\mathfrak t)$ be
the moduli space of locally free sheaves $E$ which are
$\mu_\omega$-stable and satisfy $\mathfrak t(E) = \mathfrak t$.
Define $\overline{\mathfrak U}_{\omega}(\mathfrak t)$ to be the Uhlenbeck
compactification space associated to $\omega$ and $\mathfrak t$.
\end{enumerate}
\end{definition}

\begin{lemma}  \label{00n}
Let $\mathfrak t = (0, 0, n)$ where $n \in \Z_{\ge 0}$. Then,
all the spaces $\overline{\mathfrak M}_\Omega(\mathfrak t)$,
$\overline{\mathfrak M}_{{\bf u}_m}(\mathfrak t)$ and
$\overline{\mathfrak M}_{\omega}(\mathfrak t)$ are identified
with the symmetric product $\text{\rm Sym}^n(X)$.
\end{lemma}
\begin{proof}
For $\overline{\mathfrak M}_{\omega}(\mathfrak t)$, this follows from
the fact that every $0$-dimensional torsion sheaf is generated by
the skyscraper sheaves $\mathcal O_x, x \in X$ via extensions.
For $\overline{\mathfrak M}_\Omega(\mathfrak t)$ (respectively,
$\overline{\mathfrak M}_{{\bf u}_m}(\mathfrak t)$), note that every
skyscraper sheaf $\mathcal O_x \in \mathcal P_\Omega((0, 1])$ has
phase $1$ and is $(Z_\Omega, \mathcal P_\Omega)$-stable by \cite{Bay}
(respectively, $(Z_m, \mathcal P_m)$-stable by \cite{Bri2}).
\end{proof}

\begin{lemma}  \label{0c1c2}
Let $\Omega = (\omega, \rho, p, U)$ be from
Subsect.~\ref{subsect_Large} with $U = e^{-K_X/2}$.
Let $\mathfrak t = (0, c_1, c_2)$ with $c_1 \ne 0$.
Then, $\overline{\mathfrak M}_\Omega(\mathfrak t)$ is
identified with $\overline{\mathfrak M}_{\omega}(\mathfrak t)$.
\end{lemma}
\begin{proof}
We may let $c_1 > 0$. By the proof of Lemma~4.2 in \cite{Bay}, if $E \in
\overline{\mathfrak M}_\Omega(\mathfrak t)$ is
$(Z_\Omega, \mathcal P_\Omega)$-semistable, then $E \in
\overline{\mathfrak M}_{\omega}(\mathfrak t)$ is Simpson-semistable
with respect to $\omega$.

Conversely, let $E \in \overline{\mathfrak M}_{\omega}(\mathfrak t)$
be Simpson-semistable with respect to $\omega$. Note that
\begin{eqnarray}
   \chi(E \otimes \mathcal O_X(m\omega))
&=&\,\, (c_1 \omega)m + (\ch_2 - c_1 \beta)  \label{0c1c2.1}  \\
   Z_\Omega(E)(m)
&=&i (c_1 \omega)m - (\ch_2 - c_1 \beta)   \label{0c1c2.2}
\end{eqnarray}
where $\beta = K_X/2$ and $\ch_2 = c_1^2/2 - c_2$.
Let $A$ be any proper sub-object of $E$ in $\mathcal P_\Omega((0, 1])
= \mathcal A^\sharp_{(\omega, \beta\omega)}$, and let $B = E/A$.
Then we have the exact sequence $0 \to A \to E \to B \to 0$
in $\mathcal P_\Omega((0, 1])$. Thus, $A$ is a sheaf in
$\mathcal T_{(\omega, \beta\omega)}$ sitting in
\begin{eqnarray}   \label{0c1c2.3}
0 \to \mathcal H^{-1}(B) \to A \to E \to \mathcal H^0(B) \to 0.
\end{eqnarray}
If $\mathcal H^{-1}(B) = 0$, then $A$ is a proper subsheaf of $E$.
By (\ref{0c1c2.1}),
$$
(\ch_2(A) - c_1(A) \beta)/(c_1(A)\omega) \le
(\ch_2 - c_1 \beta)/(c_1 \omega)
$$
since $E$ is Simpson-semistable with respect to $\omega$.
By (\ref{0c1c2.2}), $\phi \big (Z_\Omega(E)(m) \big )
\ge \phi \big (Z_\Omega(A)(m) \big )$ for all $m > 0$.
Assume that $\mathcal H^{-1}(B) \ne 0$. By (\ref{0c1c2.3}),
we obtain $\rk(A) = \rk \big ( \mathcal H^{-1}(B) \big ) > 0$.
So we conclude from (\ref{EB2}) that $\phi \big (Z_\Omega(E)(m) \big )
> \phi \big (Z_\Omega(A)(m) \big )$ for $m \gg 0$.
Therefore, $E \in \overline{\mathfrak M}_\Omega(\mathfrak t)$ is
$(Z_\Omega, \mathcal P_\Omega)$-semistable.
\end{proof}

\begin{definition}  \label{wall-chambers}
(see \cite{LQ})
Let $\mathbb C_X \subset \text{\rm Num}(X)_\R$ be the ample cone of
the smooth projective surface $X$. Fix a numerical type
$\mathfrak t = (r, c_1, c_2)$ on $X$.
\begin{enumerate}
\item[(i)] For a class $\xi \in \text{\rm Num}(X) \otimes \mathbb R$,
we define
\begin{eqnarray}  \label{wall-chambers.1}
W^\xi = \mathbb C_X \cap \{ \alpha \in \text{\rm Num}(X)_\R|
\,\, \alpha \cdot \xi = 0\}.
\end{eqnarray}

\item[(ii)] Let $\mathcal W(\mathfrak t)$ be the set whose elements are
of the form $W^\xi$, where $\xi$ is the numerical equivalence class
$(rF - s c_1)$ for some divisor $F$ and some integer $s$ with
$0 < s < |r|$ satisfying the inequalities:
\begin{eqnarray}           \label{wall-chambers.2}
-{r^2 \over 4} \big ( 2r c_2 - (r-1) c_1^2\big ) \le \xi^2 < 0.
\end{eqnarray}

\item[(iii)] A {\it wall of type $\mathfrak t$} is an element in
$\mathcal W(\mathfrak t)$, while a {\it chamber of type
$\mathfrak t$} is a connected component in the complement
$\mathbb C_X - \mathcal W(\mathfrak t)$.
\end{enumerate}
\end{definition}

It is well-known that the set $\mathcal W(\mathfrak t)$ of walls of
type $\mathfrak t$ is locally finite, i.e.,
given a compact subset $K$ of $\mathbb C_X$,
there are only finitely many walls $W$ of type $\mathfrak t$
such that $W \cap K \ne \emptyset$. In addition,
$\xi$ defines a wall of type $\mathfrak t$ if and only if
it defines a wall of type $(\w r, \w c_1, \w c_2)$ where $\w r = -r$
and $1 + \w c_1 + \w c_2 = (1+c_1+c_2)^{-1} \in A^*(X)$.

Fix $\mathfrak t = (r, c_1, c_2)$ with $r > 0$. Then
the Simpson-semistability is the same as the Gieseker-semistability.
If $\omega_1$ and $\omega_2$ are contained in the same chamber of
type $\mathfrak t$, then $\overline{\mathfrak M}_{\omega_1}(\mathfrak t)
= \overline{\mathfrak M}_{\omega_2}(\mathfrak t)$.
If $\omega$ is contained in a chamber of type $\mathfrak t$,
then $E$ is $\mu_\omega$-stable whenever it is $\mu_\omega$-semistable
and $\mathfrak t(E) = (r, c_1, c_2')$ with $c_2' \le c_2$. In this case,
\begin{eqnarray}   \label{uhlenbeck}
\overline{\mathfrak U}_{\omega}(r, c_1, c_2) =
\bigoplus_{c_2' \le c_2} {\mathfrak M}_{\omega}(r, c_1, c_2')
\times \text{\rm Sym}^{c_2-c_2'}(X).
\end{eqnarray}
It was proved in \cite{Li1, Li2, Mor} that
$\overline{\mathfrak U}_{\omega}(\mathfrak t)$ is a projective variety.

\section{\bf Locally finiteness of mini-walls of a fixed type}
\label{sect_Locally}

In this section, we define and study the mini-walls of a fixed type
$(\mathfrak t, \beta, \omega)$ for Bridgeland stability conditions.
We will prove that the set of the mini-walls of a fixed type
$(\mathfrak t, \beta, \omega)$ is locally finite.

\begin{definition}   \label{leading}
Let $\sigma$ be a Bridgeland stability condition on $D^b(X)$. Let
$$
0 = E_0 \subset E_1 \subset \ldots E_{n-1} \subset E_n = E
$$
be the HN-filtration of $E \in \mathcal D^b(X)$ with respect to $\sigma$.
We define $E_1$ to be the {\it leading HN-filtration component of $E$
with respect to $\sigma$}.
\end{definition}

\begin{definition}   \label{mini-walls}
Let ${\bf u} = e^{-(\beta + i \, \omega)}$ with $\omega$ being ample.
Fix a numerical type $\mathfrak t = (r, c_1, c_2)$, and fix a subset
$I$ of $(0, +\infty)$.
\begin{enumerate}
\item[(i)] A {\it mini-wall of type $(\mathfrak t, \beta,
\omega)$ in $I$} is a number $m_0 \in I$ such that
$$\phi \big (Z_{m_0}(A) \big ) = \phi \big (Z_{m_0}(E) \big )$$
where $E \in \overline{\mathfrak M}_{{\bf u}_{m_1}}(\mathfrak t)$
for some $m_1 \in I$,
$E \not \in \overline{\mathfrak M}_{{\bf u}_{m_2}}(\mathfrak t)$
for some $m_2 \in I$, and $A$ is the leading HN-filtration component
of $E$ with respect to $(Z_{m_2}, \mathcal P_{m_2})$.

\item[(ii)] A {\it mini-chamber of type $(\mathfrak t, \beta, \omega)$
in $I$} is a connected component of
$$
I \, - \, \{m|\, m \,\, \text{\rm is a mini-wall of
type $(\mathfrak t, \beta, \omega)$ in $I$} \}.
$$
\end{enumerate}
\end{definition}

\begin{remark}   \label{rmk:mini-walls}
(i) Unlike Definition~\ref{wall-chambers}~(iii), our definitions of
mini-walls and mini-chambers depends on subsets $I$ of
$(0, +\infty)$. These dependences are consistent with
the Proposition~9.3 in \cite{Bri2} where when $X$ is a $K3$ surface,
walls and chambers are defined for compact subsets
in the space of Bridgeland stability conditions.

(ii) Let $I_1 \subset I_2 \subset (0, +\infty)$. If $m$ is
a mini-wall of type $(\mathfrak t, \beta, \omega)$ in $I_1$,
then $m$ is a mini-wall of type $(\mathfrak t, \beta, \omega)$
in $I_2$. However, the converse may not be true.

(iii) Let $\mathfrak t = (0, 0, c_2)$, and let $I \subset (0, +\infty)$
be connected. By Lemma~\ref{00n}, all the spaces
$\overline{\mathfrak M}_{{\bf u}_m}(\mathfrak t)$ with $m > 0$ are
identical. Hence no mini-walls of type $(\mathfrak t, \beta, \omega)$
in $I$ exist, and the only mini-chamber of type
$(\mathfrak t, \beta, \omega)$ in $I$ is $I$ itself.
\end{remark}

\begin{lemma} \label{chamber-moduli}
Let $\mathfrak C$ be a mini-chamber of type $(\mathfrak t, \beta,
\omega)$ in $I$. If $m_1, m_2 \in \mathfrak C$, then
$$
\overline{\mathfrak M}_{{\bf u}_{m_1}}(\mathfrak t) =
\overline{\mathfrak M}_{{\bf u}_{m_2}}(\mathfrak t).
$$
\end{lemma}
\begin{proof}
By symmetry, it suffices to show that if
$E \in \overline{\mathfrak M}_{{\bf u}_{m_1}}(\mathfrak t)$, then
$E \in \overline{\mathfrak M}_{{\bf u}_{m_2}}(\mathfrak t)$. Assume that
$E \not \in \overline{\mathfrak M}_{{\bf u}_{m_2}}(\mathfrak t)$.
Let $A$ be the leading HN-filtration component of $E$ with respect
to $(Z_{m_2}, \mathcal P_{m_2})$. Then
$\phi \big (Z_{m_2}(A) \big ) > \phi \big (Z_{m_2}(E) \big )$.
Since $E \in \overline{\mathfrak M}_{{\bf u}_{m_1}}(\mathfrak t)$,
we have $\phi \big (Z_{m_2}(A) \big ) \le \phi \big (Z_{m_2}(E) \big )$.
Since $\mathfrak C$ is connected, $\phi \big (Z_{m_0}(A) \big ) =
\phi \big (Z_{m_0}(E) \big )$ for some $m_0 \in \mathfrak C$. Thus by
definition, $m_0$ is a mini-wall of type $(\mathfrak t, \beta, \omega)$
in $I$. This is impossible since the mini-chamber $\mathfrak C$ can
not contain any mini-wall.
\end{proof}

\begin{lemma} \label{quotient}
Let $\omega \in {\rm Num}(X)_\R$ be ample, and let $v \in \R$.
Let $E \in \mathcal A^\sharp_{(\omega, v)}$.

{\rm (i)} Let $f: \mathcal H^0(E) \to B$ be a surjection in
$\text{\rm Coh}(X)$ with $\ker(f), B \in \mathcal T_{(\omega, v)}$.
Then in $\mathcal A^\sharp_{(\omega, v)}$, there exists an
exact sequence of the form
\begin{eqnarray}   \label{quotient.01}
0 \to A \to E \to B \to 0.
\end{eqnarray}

{\rm (ii)} Let $g: A \to \mathcal H^{-1}(E)$ be an injection in
$\text{\rm Coh}(X)$ with $A, \text{\rm coker}(g) \in
\mathcal F_{(\omega, v)}$.
Then in $\mathcal A^\sharp_{(\omega, v)}$, there exists an
exact sequence of the form
\begin{eqnarray}   \label{quotient.02}
0 \to A[1] \to E \to B \to 0.
\end{eqnarray}
\end{lemma}
\begin{proof}
Since (ii) can be proved similarly, we will only prove (i) below.
To prove (i), note from Definition~\ref{TFA} that
$\mathcal T_{(\omega, v)} \subset \mathcal A^\sharp_{(\omega, v)}$.
So $\ker(f), \mathcal H^0(E), B \in \mathcal T_{(\omega, v)}
\subset \mathcal A^\sharp_{(\omega, v)}$. It follows that the exact sequence
$0 \to \ker(f) \to \mathcal H^0(E) \to B \to 0$ in $\text{\rm Coh}(X)$ is
an exact sequence in $\mathcal A^\sharp_{(\omega, v)}$. In particular,
we have a surjection $\mathcal H^0(E) \to B$ in $\mathcal A^\sharp_{(\omega, v)}$.
Composing with the surjection $E \to \mathcal H^0(E)$, we obtain a surjection
$E \to B$ in $\mathcal A^\sharp_{(\omega, v)}$. Letting $A$ be the kernel of
$E \to B$ in $\mathcal A^\sharp_{(\omega, v)}$, we obtain (\ref{quotient.01}).
\end{proof}

\begin{lemma}       \label{En}
Let $\omega \in {\rm Num}(X)_\R$ be ample, and let $v \in \R$.
\begin{enumerate}
\item[(i)] If $E \in \mathcal A^\sharp_{(\omega, v)}$,
then $c_1(E) \omega \ge \rk(E) v$.

\item[(ii)] Fix a numerical type $\mathfrak t = (r, c_1, c_2)$.
Let $\{E_n\}_{n=1}^{+\infty}$ be a sequence of objects in
$\mathcal A^\sharp_{(\omega, v)}$.
Assume that for each $n$, there exists an exact sequence
$$0 \to A_n \to E_n \to B_n \to 0$$ in $\mathcal A^\sharp_{(\omega, v)}$.
Then, $\lim_{n \to +\infty} c_1(A_n) \omega/\rk(A_n) = v$
if $\lim_{n \to +\infty} \rk(A_n) = \infty$, and
$\lim_{n \to +\infty} c_1(B_n) \omega/\rk(B_n) = v$
if $\lim_{n \to +\infty} \rk(B_n) = \infty$.
\end{enumerate}
\end{lemma}
\begin{proof}
(i) Note that every $E \in \mathcal A^\sharp_{(\omega, v)}$ sits in
the exact sequence
\begin{eqnarray}   \label{En.1}
0 \to \mathcal H^{-1}(E)[1] \to E \to \mathcal H^0(E) \to 0
\end{eqnarray}
in $\mathcal A^\sharp_{(\omega, v)}$. Since $\mathcal H^{-1}(E) \in
\mathcal F_{(\omega, v)}$ and $\mathcal H^0(E) \in
\mathcal T_{(\omega, v)}$, we have $c_1 \big ( \mathcal H^{-1}(E)
\big ) \omega \le \rk \big ( \mathcal H^{-1}(E) \big ) v$ and
$c_1 \big ( \mathcal H^0(E) \big ) \omega \ge \rk \big (
\mathcal H^0(E) \big ) v$. So $c_1(E) \omega \ge \rk(E) v$.

(ii) Since the second statement can be proved similarly, we will
only prove the first statement. Assume that $\lim_{n \to +\infty}
\rk(A_n) = \infty$. By (i), we have $c_1(A_n) \omega \ge \rk(A_n) v$
and $c_1(B_n) \omega \ge \rk(B_n) v$. Therefore, we conclude that
$$
\rk(A_n) v \le c_1(A_n) \omega = c_1 \omega - c_1(B_n) \omega
\le c_1 \omega - \rk(B_n) v = (c_1 \omega - rv) + \rk(A_n) v.
$$
Since $\lim_{n \to +\infty} \rk(A_n) = \infty$, it follows that
$\lim_{n \to +\infty} c_1(A_n) \omega/\rk(A_n) = v$.
\end{proof}

\begin{lemma} \label{compare-bnd}
Let $\alpha , \omega \in \text{\rm Num}(X)_\R$ with $\omega$ being
ample.
\begin{enumerate}
\item[(i)] If $c \le \alpha \cdot \omega \le d$, then $\alpha^2
\le \text{\rm max}\{c^2, d^2\}/\omega^2$.

\item[(ii)] Let $B$ be a torsion free $\mu_\omega$-semistable sheaf.
If $c \le \mu_\omega(B) \le d$, then $c(B)/\rk(B)$ is bounded from
below by a constant depending only on $c, d, \omega, \beta$.
\end{enumerate}
\end{lemma}
\begin{proof}
(i) Write $\alpha = a \omega + \rho$ with $\rho \cdot \omega = 0$.
Then, $a = (\alpha \cdot \omega)/\omega^2$. By the Hodge Index
Theorem, we have $\rho^2 \le 0$. It follows that
\begin{eqnarray*}
\alpha^2 = a^2 \omega^2 + \rho^2
\le a^2 \omega^2 = (\alpha \cdot \omega)^2/\omega^2
\le \text{\rm max}\{c^2, d^2\}/\omega^2.
\end{eqnarray*}

(ii) Since $c \le \mu_\omega(B) \le d$, we obtain the inequalities
$$
c - \beta \cdot \omega \le
\left ( {c_1(B) \over \rk(B)} - \beta \right ) \cdot \omega
\le d - \beta \cdot \omega.
$$
By the Bogomolov Inequality, $\ch_2(B) \le c_1(B)^2/(2 \, \rk(B))$.
By (\ref{cE}),
\begin{eqnarray*}
  {c(B) \over \rk(B)}
= -{\ch_2(B) \over \rk(B)} + {c_1(B) \over \rk(B)}
     \cdot \beta - {\beta^2 \over 2}
\ge -{1 \over 2} \, \left ( {c_1(B) \over
     \rk(B)} - \beta \right )^2.
\end{eqnarray*}
Now our conclusion about $c(B)/\rk(B)$ follows immediately from (i).
\end{proof}

\begin{proposition}       \label{bnd-H01}
Let ${\bf u} = e^{-(\beta + i \, \omega)}$ with $\omega$ being
ample and $\beta, \omega \in \text{\rm Num}(X)_\R$.
Fix a numerical type $\mathfrak t = (r, c_1, c_2)$
and an interval $I = [a, +\infty)$ with $a > 0$. Then there
exists a positive number $N$ depending only on $\mathfrak t$, $\beta$,
$\omega$ and $I$ such that $\rk \big ( \mathcal H^0(E) \big )$,
$\rk \big ( \mathcal H^{-1}(E) \big ) \le N$ whenever $E \in
\overline{\mathfrak M}_{{\bf u}_m}(\mathfrak t)$ for some $m \in I$.
\end{proposition}
\begin{proof}
Since $\rk \big ( \mathcal H^{-1}(E)
\big ) = \rk \big ( \mathcal H^0(E) \big ) - r$, it suffices to prove
the lemma for $\rk \big ( \mathcal H^0(E) \big )$. Assume that such
an $N$ does not exist for $\rk \big ( \mathcal H^0(E) \big )$.
Then there exists a sequence of objects $E_n
\in \overline{\mathfrak M}_{{\bf u}_{m_n}}(\mathfrak t)$,
$n = 1, 2, \ldots$, with $m_n \in I$ for all $n$ and
$\lim_{n \to +\infty} \rk \big ( \mathcal H^0(E_n) \big ) = +\infty$.
Replacing by a subsequence if necessary, we may further assume
that $\lim_{n \to +\infty} m_n = m_0$ (possibly $+\infty$).

To draw a contradiction, let $A_{n, 0} = \mathcal H^0(E_n)$ and
$A_{n, 1} = \mathcal H^{-1}(E_n)$. Then we have the exact sequence
$0 \to A_{n, 1}[-1] \to E_n \to A_{n, 0} \to 0$ in
$\mathcal A^\sharp_{(\omega, \beta\omega)}$. By Lemma~\ref{En}~(ii),
$\lim_{n \to +\infty} \mu_\omega(A_{n, 0}) = \beta \omega$.
Let $F_n = A_{n, 0}/\text{\rm Tor}(A_{n, 0})$, and let
$F_n^{(1)}, \, \ldots, \, F_n^{(\ell_n)}$ be the usual
HN-filtration quotients of $F_n$ with respect to $\mu_\omega$ such
that $\mu_\omega \big ( F_n^{(1)} \big ) > \ldots >
\mu_\omega \big ( F_n^{(\ell_n)} \big )$.
Then $F_n^{(1)}, \ldots, F_n^{(\ell_n)}$ are torsion free
and $\mu_\omega$-semistable. Moreover, $\beta\omega < \mu_\omega
\big ( F_n^{(\ell_n)} \big ) \le \mu_\omega(F_n) \le
\mu_\omega(A_{n, 0})$. Thus $\lim_{n \to +\infty} \mu_\omega \big (
F_n^{(\ell_n)} \big ) = \beta \omega$. In particular, for $n \gg 0$,
we have $\beta\omega < \mu_\omega \big ( F_n^{(\ell_n)} \big ) \le
\beta\omega + \epsilon_n$ where $\{\epsilon_n\}_{n \gg 0}$ is
a sequence of positive numbers with $\lim_{n \to +\infty} \epsilon_n
= 0$. By the proof of Lemma~\ref{compare-bnd},
\begin{eqnarray}   \label{bnd-H01.1}
  {c\big ( F_n^{(\ell_n)} \big ) \over
       \rk\big ( F_n^{(\ell_n)} \big )}
\ge -{1 \over 2} \, \left ( {c_1\big ( F_n^{(\ell_n)} \big ) \over
     \rk\big ( F_n^{(\ell_n)} \big )} - \beta \right )^2
\ge -{\epsilon_n^2 \over 2 \omega^2}.
\end{eqnarray}

On the other hand, by Lemma~\ref{quotient}~(i), there exists
a quotient $E_n \to F_n^{(\ell_n)} \to 0$ in
$\mathcal A^\sharp_{(\omega, \beta\omega)}$.
Since $E_n \in \overline{\mathfrak M}_{{\bf u}_{m_n}}(\mathfrak t)$,
we see from (\ref{EB2}) that
\begin{eqnarray*}
& &{\omega^2 m_n^2 \over 2} \big ( r \, c_1(F_n^{(\ell_n)})
   \omega - \rk(F_n^{(\ell_n)}) \, c_1 \omega \big )       \\
&\ge&c(F_n^{(\ell_n)}) \, \big ( c_1 \omega - r \, \beta \omega \big )
   - c(\mathfrak t, \beta) \, \big ( c_1(F_n^{(\ell_n)})
   \omega - \rk(F_n^{(\ell_n)}) \, \beta \omega \big )
\end{eqnarray*}
where $c(\mathfrak t, \beta) = c(E_n)$ depends only on
$\mathfrak t$ and $\beta$. By Lemma~\ref{En}~(i), $c_1 \omega -
r \, \beta \omega \ge 0$. If $c_1 \omega - r \, \beta \omega = 0$,
then the phase of $E_n$ with respect to every charge
$Z_{{\bf u}_m}, m > 0$ is equal to $1$.
So $E_n \in \overline{\mathfrak M}_{{\bf u}_m}(\mathfrak t)$
for all $n, m > 0$. In view of Lemma~\ref{4.1} and Lemma~\ref{4.2},
this contradicts to
$\lim_{n \to +\infty} \rk \big ( \mathcal H^0(E_n) \big ) = +\infty$.
Thus $c_1 \omega - r \, \beta \omega > 0$ and
\begin{eqnarray*}
   {c(F_n^{(\ell_n)}) \over \rk(F_n^{(\ell_n)})}
\le {\omega^2 m_n^2 \over 2} \cdot {r \, \mu_\omega(F_n^{(\ell_n)})
   - c_1 \omega \over c_1 \omega - r \, \beta \omega}
   + c(\mathfrak t, \beta) \cdot  {\mu_\omega(F_n^{(\ell_n)})
   - \beta\omega \over c_1 \omega - r \, \beta \omega}.
\end{eqnarray*}
Combining this with (\ref{bnd-H01.1}), we conclude that
\begin{eqnarray*}
{\omega^2 m_n^2 \over 2} \cdot {r \, \mu_\omega(F_n^{(\ell_n)})
   - c_1 \omega \over c_1 \omega - r \, \beta \omega}
   + c(\mathfrak t, \beta) \cdot  {\mu_\omega(F_n^{(\ell_n)})
   - \beta\omega \over c_1 \omega - r \, \beta \omega}
\ge -{\epsilon_n^2 \over 2 \omega^2}.
\end{eqnarray*}
Recall that $\lim_{n \to +\infty} \mu_\omega \big ( F_n^{(\ell_n)}
\big ) = \beta \omega$ and $\lim_{n \to +\infty} m_n = m_0 \ge a > 0$.
Letting $n \to +\infty$, we see that $-\omega^2 m_0^2/2 \ge 0$
which is impossible.
\end{proof}

\begin{lemma}       \label{bnd-rk}
Let ${\bf u} = e^{-(\beta + i \, \omega)}$ with $\omega$ being
ample and $\beta, \omega \in \text{\rm Num}(X)_\R$. Fix a numerical
type $\mathfrak t = (r, c_1, c_2)$ and an interval $I = [a, +\infty)$
with $a > 0$. Then there exists $N$ depending
only on $\mathfrak t$, $\beta$, $\omega$ and $I$ such that $|\rk(A)|
\le N$ whenever
\begin{enumerate}
\item[$\bullet$] $A$ is $(Z_{m_2}, \mathcal P_{m_2})$-semistable for
      some $m_2 \in I$;
\item[$\bullet$] $A$ is a proper sub-object of certain object
      $E \in \mathcal A^\sharp_{(\omega, \beta\omega)}$
      with $\mathfrak t(E) = \mathfrak t$;
\item[$\bullet$] $A$ destablizes $E$ with respect to
      $(Z_{m_2}, \mathcal P_{m_2})$;
\item[$\bullet$] $\phi \big (Z_{m_{0}}(A) \big ) =
      \phi \big (Z_{m_{0}}(E) \big )$ for some $m_0 \in I$.
\end{enumerate}
\end{lemma}
\begin{proof}
Assume that our statement is not true.
Then there exists a sequence of sub-objects $A_n \subset E_n
(n = 1, 2, \ldots)$ in $\mathcal A^\sharp_{(\omega, \beta\omega)}$
such that $\lim_{n \to +\infty} \rk(A_n) = \pm \infty$,
$\mathfrak t(E_n) = \mathfrak t$, $A_n$ is $(Z_{m_{2, n}},
\mathcal P_{m_{2, n}})$-semistable for some $m_{2, n} \in I$,
$A_n$ destablizes $E_n$ with respect to $(Z_{m_{2, n}},
\mathcal P_{m_{2, n}})$, and $\phi \big (Z_{m_{0, n}}(A_n) \big )
= \phi \big (Z_{m_{0, n}}(E_n) \big )$ for some $m_{0, n} \in I$.
We may assume that $\lim_{n \to +\infty} m_{i, n} = m_i$
(possibly $+\infty$) for $i = 0, 2$. Define
\begin{eqnarray}   \label{dtboA}
d(\mathfrak t, \beta, \omega, A_n)
= c(A_n) \, \big ( c_1 \omega - r \, \beta \omega \big )
- c(\mathfrak t, \beta) \, \big ( c_1(A_n) \omega -
      \rk(A_n) \, \beta \omega \big )
\end{eqnarray}
where $c(\mathfrak t, \beta) = c(E_n)$ depends only on $\mathfrak t$
and $\beta$. Since $\phi \big (Z_{m_{0, n}}(A_n) \big )
= \phi \big (Z_{m_{0, n}}(E_n) \big )$
and $\phi \big (Z_{m_{2, n}}(A_n) \big ) >
\phi \big (Z_{m_{2, n}}(E_n) \big )$,
we see from (\ref{EB2}) that
\begin{eqnarray}
  {\omega^2 m_{0, n}^2 \over 2} \big ( r \, c_1(A_n) \omega -
    \rk(A_n) \, c_1 \omega \big )
&=&d(\mathfrak t, \beta, \omega, A_n),    \label{bnd-rk.1}  \\
  {\omega^2 m_{2, n}^2 \over 2} \big ( r \, c_1(A_n) \omega -
    \rk(A_n) \, c_1 \omega \big )
&>&d(\mathfrak t, \beta, \omega, A_n).      \label{bnd-rk.112}
\end{eqnarray}
Since $E_n$ is not $(Z_{m_{2, n}}, \mathcal P_{m_{2, n}})$-semistable,
we must have $\phi \big (Z_{m_{2, n}}(E_n) \big ) < 1$.
So $c_1 \omega - r \, \beta \omega > 0$. This in turn implies that
$\phi \big (Z_{m_{0, n}}(A_n) \big )
= \phi \big (Z_{m_{0, n}}(E_n) \big ) < 1$. Therefore,
$c_1(A_n) \omega - \rk(A_n) \beta \omega > 0$. In summary, we obtain
\begin{eqnarray}   \label{bnd-rk.2}
c_1 \omega - r \, \beta \omega > 0, \,\,
c_1(A_n) \omega - \rk(A_n) \beta \omega  > 0.
\end{eqnarray}
By Lemma~\ref{En}~(ii),
$\lim_{n \to +\infty} \mu_\omega(A_n) = \beta\omega$. Dividing both
sides of (\ref{bnd-rk.1}) by $\rk(A_n) \, (c_1 \omega -
r \, \beta \omega)$ and using (\ref{dtboA}), we conclude that
\begin{eqnarray}       \label{cAn-rkAn}
{c(A_n) \over \rk(A_n)}
= {\omega^2 m_{0, n}^2 \over 2} \cdot {r \, \mu_\omega(A_n) -
  c_1 \omega \over c_1 \omega - r \, \beta\omega}
  + c(\mathfrak t, \beta) \cdot {\mu_\omega(A_n) - \beta\omega
  \over c_1 \omega - r \, \beta\omega}.
\end{eqnarray}
Let $A_{n, 0} = \mathcal H^0(A_n)$, $A_{n, 1} =
\mathcal H^{-1}(A_n)$ and $F_n = A_{n, 0}/\text{\rm Tor}(A_{n, 0})$.

\medskip
{\bf Case 1: $\lim_{n \to +\infty} \rk(A_n) = +\infty$}. Then
$\lim_{n \to +\infty} \rk \big ( A_{n, 0} \big ) = +\infty$.
Since $A_{n, 1} \in \mathcal F_{(\omega, \beta\omega)}$, we get
$c_1(A_{n, 1})\omega \le \rk(A_{n, 1}) \, \beta \omega$.
Since $F_n \in \mathcal T_{(\omega, \beta\omega)}$,
\begin{eqnarray*}
    \beta \omega
&<&\mu_\omega(F_n) \le \mu_\omega(A_{n, 0})
  = {c_1(A_n) \omega + c_1(A_{n, 1})\omega \over \rk(A_{n, 0})} \\
&\le&{c_1(A_n) \omega + \rk(A_{n, 1}) \, \beta \omega \over \rk(A_{n, 0})}
  ={c_1(A_n) \omega + [\rk(A_{n, 0}) - \rk(A_n)] \, \beta \omega
  \over \rk(A_{n, 0})}   \\
&=&\beta \omega + {\rk(A_n) \over \rk(A_{n, 0})} \cdot
  (\mu_\omega(A_n) - \beta \omega)
  \le \beta \omega + (\mu_\omega(A_n) - \beta \omega)
  = \mu_\omega(A_n).
\end{eqnarray*}
So $\lim_{n \to +\infty} \mu_\omega(F_n) = \beta\omega$.
Let $F_n^{(1)}, \, \ldots, \, F_n^{(\ell_n)}$ be the usual
HN-filtration quotients of $F_n$ with respect to $\mu_\omega$ such
that $\mu_\omega \big ( F_n^{(1)} \big ) > \ldots >
\mu_\omega \big ( F_n^{(\ell_n)} \big )$.
Then $F_n^{(1)}, \ldots, F_n^{(\ell_n)}$ are torsion free
and $\mu_\omega$-semistable. Moreover, $\beta\omega < \mu_\omega
\big ( F_n^{(\ell_n)} \big ) \le \mu_\omega(F_n) \le \mu_\omega(A_n)$.
Thus $\lim_{n \to +\infty} \mu_\omega \big (
F_n^{(\ell_n)} \big ) = \beta \omega$, and
$\beta\omega < \mu_\omega \big ( F_n^{(\ell_n)} \big ) \le
\beta\omega + \epsilon_n$ for $n \gg 0$ where $\{\epsilon_n\}_{n \gg 0}$
is a sequence of positive numbers with $\lim_{n \to +\infty}
\epsilon_n = 0$. As in (\ref{bnd-H01.1}),
\begin{eqnarray}   \label{bnd-rk.4}
  {c\big ( F_n^{(\ell_n)} \big ) \over
       \rk\big ( F_n^{(\ell_n)} \big )}
\ge -{\epsilon_n^2 \over 2 \omega^2}.
\end{eqnarray}
On the other hand, by Lemma~\ref{quotient}~(i), there exists a quotient
$A_n \to F_n^{(\ell_n)} \to 0$ in $\mathcal A^\sharp_{(\omega, \beta\omega)}$.
Since $A_n$ is $(Z_{m_{2, n}}, \mathcal P_{m_{2, n}})$-semistable, we have
\begin{eqnarray*}
& &{\omega^2 m_{2, n}^2 \over 2} \big ( \rk(A_n) \, c_1(F_n^{(\ell_n)})
   \omega - \rk(F_n^{(\ell_n)}) \, c_1(A_n) \omega \big )       \\
&\ge&c(F_n^{(\ell_n)}) \, \big ( c_1(A_n) \omega -
   \rk(A_n) \, \beta \omega \big ) - c(A_n) \, \big ( c_1(F_n^{(\ell_n)})
   \omega - \rk(F_n^{(\ell_n)}) \, \beta \omega \big ).
\end{eqnarray*}
by (\ref{EB2}). Since $c_1(A_n) \omega - \rk(A_n) \, \beta \omega > 0$,
we obtain
\begin{eqnarray*}
   {c(F_n^{(\ell_n)}) \over \rk(F_n^{(\ell_n)})}
&\le&{c(A_n) \over \rk(A_n)} \cdot  {\mu_\omega(F_n^{(\ell_n)})
   - \beta\omega \over \mu_\omega(A_n) - \beta\omega} +
   {\omega^2 m_{2, n}^2 \over 2} \cdot {\mu_\omega(F_n^{(\ell_n)})
   - \mu_\omega(A_n) \over \mu_\omega(A_n) - \beta\omega}       \\
&=&{\omega^2 m_{0, n}^2 \over 2} \cdot {r \mu_\omega(A_n) - c_1 \omega
   \over c_1 \omega - r \, \beta\omega} \cdot {\mu_\omega(F_n^{(\ell_n)})
   - \beta\omega \over \mu_\omega(A_n) - \beta\omega}
   \\
& & + \, c(\mathfrak t, \beta) \cdot {\mu_\omega(F_n^{(\ell_n)}) -
   \beta\omega \over c_1 \omega - r \, \beta\omega}
   + {\omega^2 m_{2, n}^2 \over 2} \cdot {\mu_\omega(F_n^{(\ell_n)})
   - \mu_\omega(A_n) \over \mu_\omega(A_n) - \beta\omega}
\end{eqnarray*}
where we have used (\ref{cAn-rkAn}) in the second step. Combining with
(\ref{bnd-rk.4}), we get
\begin{eqnarray}   \label{combining}
   -{\epsilon_n^2 \over 2 \omega^2}
&\le&{\omega^2 m_{0, n}^2 \over 2} \cdot {r \mu_\omega(A_n) - c_1 \omega
   \over c_1 \omega - r \, \beta\omega} \cdot {\mu_\omega(F_n^{(\ell_n)})
   - \beta\omega \over \mu_\omega(A_n) - \beta\omega}   \nonumber  \\
& & + \, {\omega^2 m_{2, n}^2 \over 2} \cdot {\mu_\omega(F_n^{(\ell_n)})
   - \mu_\omega(A_n) \over \mu_\omega(A_n) - \beta\omega}
   + c(\mathfrak t, \beta) \cdot {\mu_\omega(F_n^{(\ell_n)}) -
   \beta\omega \over c_1 \omega - r \, \beta\omega}. \quad
\end{eqnarray}

Note that we may assume either $m_{0, n} < m_{2, n}$ for all $n$ or
$m_{0, n} > m_{2, n}$ for all $n$. If $m_{0, n} < m_{2, n}$ for all $n$,
then since $\mu_\omega(F_n^{(\ell_n)}) - \mu_\omega(A_n) \le 0$,
we see from (\ref{combining}) that
\begin{eqnarray*}
   -{\epsilon_n^2 \over 2 \omega^2}
&\le&{\omega^2 m_{0, n}^2 \over 2} \cdot {r \mu_\omega(A_n) - c_1 \omega
   \over c_1 \omega - r \, \beta\omega} \cdot {\mu_\omega(F_n^{(\ell_n)})
   - \beta\omega \over \mu_\omega(A_n) - \beta\omega}   \nonumber  \\
& & + \, {\omega^2 m_{0, n}^2 \over 2} \cdot {\mu_\omega(F_n^{(\ell_n)})
   - \mu_\omega(A_n) \over \mu_\omega(A_n) - \beta\omega}
   + c(\mathfrak t, \beta) \cdot {\mu_\omega(F_n^{(\ell_n)}) -
   \beta\omega \over c_1 \omega - r \, \beta\omega}   \\
&=&{\omega^2 m_{0, n}^2 \over 2} \cdot {r \mu_\omega(F_n^{(\ell_n)})
   - c_1 \omega \over c_1 \omega - r \, \beta\omega}
   + c(\mathfrak t, \beta) \cdot {\mu_\omega(F_n^{(\ell_n)}) -
   \beta\omega \over c_1 \omega - r \, \beta\omega}.
\end{eqnarray*}
Letting $n \to +\infty$, we obtain $0 \le -\omega^2 m_0^2/2$ which is
impossible since $m_0 \ge a > 0$. Similarly, if $m_{0, n} > m_{2, n}$ for
all $n$, then $r \, \mu_\omega(A_n) - c_1 \omega < 0$ by (\ref{bnd-rk.1})
and (\ref{bnd-rk.112}). Therefore, we conclude from (\ref{combining})
again that
\begin{eqnarray*}
   -{\epsilon_n^2 \over 2 \omega^2}
&<&{\omega^2 m_{2, n}^2 \over 2} \cdot {r \mu_\omega(A_n) - c_1 \omega
   \over c_1 \omega - r \, \beta\omega} \cdot {\mu_\omega(F_n^{(\ell_n)})
   - \beta\omega \over \mu_\omega(A_n) - \beta\omega}   \nonumber  \\
& & + \, {\omega^2 m_{2, n}^2 \over 2} \cdot {\mu_\omega(F_n^{(\ell_n)})
   - \mu_\omega(A_n) \over \mu_\omega(A_n) - \beta\omega}
   + c(\mathfrak t, \beta) \cdot {\mu_\omega(F_n^{(\ell_n)}) -
   \beta\omega \over c_1 \omega - r \, \beta\omega}   \\
&=&{\omega^2 m_{2, n}^2 \over 2} \cdot {r \mu_\omega(F_n^{(\ell_n)})
   - c_1 \omega \over c_1 \omega - r \, \beta\omega}
   + c(\mathfrak t, \beta) \cdot {\mu_\omega(F_n^{(\ell_n)}) -
   \beta\omega \over c_1 \omega - r \, \beta\omega}.
\end{eqnarray*}
Letting $n \to +\infty$, we obtain $0 \le -\omega^2 m_2^2/2$ which is
impossible since $m_2 \ge a > 0$.

\medskip
{\bf Case 2: $\lim_{n \to +\infty} \rk(A_n) = -\infty$}. Then
$\lim_{n \to +\infty} \rk \big ( A_{n, 1} \big ) = +\infty$ and
\begin{eqnarray*}
    \beta \omega
&\ge&\mu_\omega(A_{n, 1})
  = {c_1(A_{n, 0})\omega - c_1(A_n) \omega \over \rk(A_{n, 1})}
  \ge {\rk(A_{n, 0}) \, \beta \omega - c_1(A_n) \omega \over \rk(A_{n, 1})} \\
&=&{[\rk(A_{n, 1}) + \rk(A_n)] \, \beta \omega
  -c_1(A_n) \omega \over \rk(A_{n, 1})}
  = \beta \omega - {\rk(A_n) \over \rk(A_{n, 1})} \cdot
  (\mu_\omega(A_n) - \beta \omega)    \\
&\ge&\beta \omega + (\mu_\omega(A_n) - \beta \omega) = \mu_\omega(A_n).
\end{eqnarray*}
So $\lim_{n \to +\infty} \mu_\omega(A_{n, 1}) = \beta\omega$.
Let $G_n^{(1)}, \, \ldots, \, G_n^{(k_n)}$ be the usual
HN-filtration quotients of $A_{n, 1}$ with respect to $\mu_\omega$ such
that $\mu_\omega \big ( G_n^{(1)} \big ) > \ldots >
\mu_\omega \big ( G_n^{(k_n)} \big )$.
Then $G_n^{(1)}$ is torsion free
and $\mu_\omega$-semistable. Moreover, $\beta\omega \ge \mu_\omega
\big ( G_n^{(1)} \big ) \ge \mu_\omega(A_{n, 1}) \ge \mu_\omega(A_n)$.
Thus $\lim_{n \to +\infty} \mu_\omega \big ( G_n^{(1)} \big ) =
\beta \omega$, and $\beta\omega \ge \mu_\omega \big ( G_n^{(1)} \big ) \ge
\beta\omega - \epsilon_n$ for $n \gg 0$ where $\{\epsilon_n\}_{n \gg 0}$
is a sequence of positive numbers with $\lim_{n \to +\infty}
\epsilon_n = 0$. As in (\ref{bnd-H01.1}),
\begin{eqnarray}   \label{bnd-rk.400}
  {c\big ( G_n^{(1)} \big ) \over \rk\big ( G_n^{(1)} \big )}
\ge -{\epsilon_n^2 \over 2 \omega^2}.
\end{eqnarray}
On the other hand, by Lemma~\ref{quotient}~(ii), there exists an injection
$0 \to G_n^{(1)}[1] \to A_n$ in $\mathcal A^\sharp_{(\omega, \beta\omega)}$.
Since $A_n$ is $(Z_{m_{2, n}}, \mathcal P_{m_{2, n}})$-semistable, we have
\begin{eqnarray*}
& &-{\omega^2 m_{2, n}^2 \over 2} \big ( \rk(A_n) \, c_1(G_n^{(1)})
   \omega - \rk(G_n^{(1)}) \, c_1(A_n) \omega \big )       \\
&\le&-c(G_n^{(1)}) \, \big ( c_1(A_n) \omega -
   \rk(A_n) \, \beta \omega \big ) + c(A_n) \, \big ( c_1(G_n^{(1)})
   \omega - \rk(G_n^{(1)}) \, \beta \omega \big ).
\end{eqnarray*}
by (\ref{EB2}). Since $c_1(A_n) \omega - \rk(A_n) \, \beta \omega > 0$,
we obtain
\begin{eqnarray*}
   {c(G_n^{(1)}) \over \rk(G_n^{(1)})}
&\le&{c(A_n) \over \rk(A_n)} \cdot  {\mu_\omega(G_n^{(1)})
   - \beta\omega \over \mu_\omega(A_n) - \beta\omega} +
   {\omega^2 m_{2, n}^2 \over 2} \cdot {\mu_\omega(G_n^{(1)})
   - \mu_\omega(A_n) \over \mu_\omega(A_n) - \beta\omega}       \\
&=&{\omega^2 m_{0, n}^2 \over 2} \cdot {r \mu_\omega(A_n) - c_1 \omega
   \over c_1 \omega - r \, \beta\omega} \cdot {\mu_\omega(G_n^{(1)})
   - \beta\omega \over \mu_\omega(A_n) - \beta\omega}     \\
& & + \, c(\mathfrak t, \beta) \cdot {\mu_\omega(G_n^{(1)}) -
   \beta\omega \over c_1 \omega - r \, \beta\omega}
   + {\omega^2 m_{2, n}^2 \over 2} \cdot {\mu_\omega(G_n^{(1)})
   - \mu_\omega(A_n) \over \mu_\omega(A_n) - \beta\omega}
\end{eqnarray*}
where we have used (\ref{cAn-rkAn}) in the second step. Combining with
(\ref{bnd-rk.400}), we get
\begin{eqnarray*}
   -{\epsilon_n^2 \over 2 \omega^2}
&\le&{\omega^2 m_{0, n}^2 \over 2} \cdot {r \mu_\omega(A_n) - c_1 \omega
   \over c_1 \omega - r \, \beta\omega} \cdot {\mu_\omega(G_n^{(1)})
   - \beta\omega \over \mu_\omega(A_n) - \beta\omega}   \nonumber  \\
& & + \, c(\mathfrak t, \beta) \cdot {\mu_\omega(G_n^{(1)}) -
   \beta\omega \over c_1 \omega - r \, \beta\omega}
   + {\omega^2 m_{2, n}^2 \over 2} \cdot {\mu_\omega(G_n^{(1)})
   - \mu_\omega(A_n) \over \mu_\omega(A_n) - \beta\omega}. \quad
\end{eqnarray*}

If $m_{0, n} > m_{2, n}$ for all $n$,
then since $\mu_\omega(G_n^{(1)}) - \mu_\omega(A_n) \ge 0$, we have
\begin{eqnarray*}
   -{\epsilon_n^2 \over 2 \omega^2}
&\le&{\omega^2 m_{0, n}^2 \over 2} \cdot {r \mu_\omega(A_n) - c_1 \omega
   \over c_1 \omega - r \, \beta\omega} \cdot {\mu_\omega(G_n^{(1)})
   - \beta\omega \over \mu_\omega(A_n) - \beta\omega}   \nonumber  \\
& & + \, {\omega^2 m_{0, n}^2 \over 2} \cdot {\mu_\omega(G_n^{(1)})
   - \mu_\omega(A_n) \over \mu_\omega(A_n) - \beta\omega}
   + c(\mathfrak t, \beta) \cdot {\mu_\omega(G_n^{(1)}) -
   \beta\omega \over c_1 \omega - r \, \beta\omega}   \\
&=&{\omega^2 m_{0, n}^2 \over 2} \cdot {r \mu_\omega(G_n^{(1)})
   - c_1 \omega \over c_1 \omega - r \, \beta\omega}
   + c(\mathfrak t, \beta) \cdot {\mu_\omega(G_n^{(1)}) -
   \beta\omega \over c_1 \omega - r \, \beta\omega}.
\end{eqnarray*}
Letting $n \to +\infty$, we obtain the contradiction $0 \le -\omega^2 m_0^2/2$.
Similarly, if $m_{0, n} < m_{2, n}$ for all $n$,
then $r \, \mu_\omega(A_n) - c_1 \omega > 0$ by (\ref{bnd-rk.1}) and
(\ref{bnd-rk.112}). Therefore,
\begin{eqnarray*}
   -{\epsilon_n^2 \over 2 \omega^2}
&\le&{\omega^2 m_{2, n}^2 \over 2} \cdot {r \mu_\omega(A_n) - c_1 \omega
   \over c_1 \omega - r \, \beta\omega} \cdot {\mu_\omega(G_n^{(1)})
   - \beta\omega \over \mu_\omega(A_n) - \beta\omega}   \nonumber  \\
& & + \, {\omega^2 m_{2, n}^2 \over 2} \cdot {\mu_\omega(G_n^{(1)})
   - \mu_\omega(A_n) \over \mu_\omega(A_n) - \beta\omega}
   + c(\mathfrak t, \beta) \cdot {\mu_\omega(G_n^{(1)}) -
   \beta\omega \over c_1 \omega - r \, \beta\omega}   \\
&=&{\omega^2 m_{2, n}^2 \over 2} \cdot {r \mu_\omega(G_n^{(1)})
   - c_1 \omega \over c_1 \omega - r \, \beta\omega}
   + c(\mathfrak t, \beta) \cdot {\mu_\omega(G_n^{(1)}) -
   \beta\omega \over c_1 \omega - r \, \beta\omega}.
\end{eqnarray*}
Again, letting $n \to +\infty$, we obtain the contradiction $0 \le
-\omega^2 m_2^2/2$.
\end{proof}

\begin{proposition}       \label{prop-loc-fnt}
The set of mini-walls is locally finite. More precisely,
fix $\beta, \omega \in \text{\rm Num}(X)_\Q$ with $\omega$ being ample,
$\mathfrak t = (r, c_1, c_2)$, and $I = [a, b]$ with $0 < a < b$.
Then there exist only finitely many mini-walls of type $(\mathfrak t,
\beta, \omega)$ in $I$.
\end{proposition}
\begin{proof}
We may assume that $\beta, \omega \in \text{\rm Num}(X)$. Let ${\bf u} =
e^{-(\beta + i \, \omega)}$, and let $m_0$ be a mini-wall of type
$(\mathfrak t, \beta, \omega)$ in $I$.
Then $\phi \big (Z_{m_0}(A) \big ) = \phi \big (Z_{m_0}(E) \big )$
where $E \in \overline{\mathfrak M}_{{\bf u}_{m_1}}(\mathfrak t)$
for some $m_1 \in I$,
$E \not \in \overline{\mathfrak M}_{{\bf u}_{m_2}}(\mathfrak t)$
for some $m_2 \in I$, and $A$ is the leading HN-filtration component
of $E$ with respect to $(Z_{m_2}, \mathcal P_{m_2})$.
So $A$ is $(Z_{m_2}, \mathcal P_{m_2})$-semistable.
By Lemma~\ref{bnd-rk}, $|\rk(A)| \le N$ where $N$ depends only on
$\mathfrak t$, $\beta$, $\omega$ and $I$.

Since $\phi \big (Z_{m_0}(A) \big ) = \phi \big (Z_{m_0}(E) \big )$
and $A$ destablizes $E$ with respect to $(Z_{m_2}, \mathcal P_{m_2})$,
\begin{eqnarray*}
  {\omega^2 m_0^2 \over 2} \big ( r \, c_1(A) \omega -
    \rk(A) \, c_1 \omega \big )
&=&d(\mathfrak t, \beta, \omega, A)      \\
  {\omega^2 m_2^2 \over 2} \big ( r \, c_1(A) \omega -
    \rk(A) \, c_1 \omega \big )
&>&d(\mathfrak t, \beta, \omega, A).
\end{eqnarray*}
So $(r \, c_1(A) \omega - \rk(A) \, c_1 \omega) \ne 0$, and $m_0^2$
is equal to the rational number
\begin{eqnarray*}
{2 \, d(\mathfrak t, \beta, \omega, A) \over \omega^2 \cdot
\big ( r \, c_1(A) \omega - \rk(A) \, c_1 \omega \big )}
\quad \in [a^2, b^2].
\end{eqnarray*}
To prove that there are only finitely many choices for $m_0$,
it suffices to show that the positive integer
$|r \, c_1(A) \omega - \rk(A) \, c_1 \omega|$ from the
denominator is bounded from above by a number depending only
on $\mathfrak t$, $\beta$, $\omega$ and $I$. Since $|\rk(A)| \le N$,
it remains to prove that there exist $N_1$ and $N_2$ depending only
on $\mathfrak t$, $\beta$, $\omega$ and $I$ such that
\begin{eqnarray}  \label{prop-loc-fnt.1}
r \, N_1 \le r \, c_1(A) \omega \le r \, N_2.
\end{eqnarray}

Put $B = E/A$. Since $A, B \in \mathcal A^\sharp_{(\omega,
\beta\omega)}$, we see from Lemma~\ref{En}~(i) that
\begin{eqnarray}   \label{prop-loc-fnt.2}
c_1(A) \omega \ge \rk(A) \beta \omega, \qquad
c_1(B) \omega \ge \rk(B) \beta \omega.
\end{eqnarray}
Note that (\ref{prop-loc-fnt.1}) is trivially true if $r = 0$.
If $r < 0$, then by (\ref{prop-loc-fnt.2}),
\begin{eqnarray*}
   r \, c_1(A) \omega
&=&r \, c_1 \omega - r \, c_1(B) \omega
     \ge r \, c_1 \omega - r \, \rk(B) \, \beta \omega  \\
&=&r \, c_1 \omega - r \, (r - \rk(A)) \, \beta \omega
     \ge r \, c_1 \omega - |r| \, (|r| + N) \, |\beta \omega|.
\end{eqnarray*}
In addition, we have $r \, c_1(A) \omega \le r\, \rk(A) \, \beta
\omega \le |r| \, N \, |\beta \omega|$. Therefore,
(\ref{prop-loc-fnt.1}) holds for $r < 0$. Similarly, we see that
(\ref{prop-loc-fnt.1}) holds for $r > 0$ as well.
\end{proof}

We remark that when $I = [a, +\infty)$ with $a > 0$, the proof of
Proposition~\ref{prop-loc-fnt} does not go through since
it is unclear how to bound $|2\, d(\mathfrak t, \beta,
\omega, A)|$ from above.

\section{\bf Identify $\overline{\mathfrak M}_\Omega(\mathfrak t)$
and $\overline{\mathfrak M}_{{\bf u}_m}(\mathfrak t)$ for $m \gg 0$}
\label{sect_Identify}

In this section, we will strength Lemma 2.6. We show that there exists a constant
$M$ depending only on $\mathfrak t(E), \omega$ and $\beta$ such that
$E \in D^b(X)$ is $(Z_\Omega, \mathcal P_\Omega)$-semistable if and only if
$E$ is $(Z_m, \mathcal P_m)$-semistable for some $m \ge M$.

\begin{definition} \label{def-universal}
If $E \in D^b(X)$ and $\beta, \omega \in {\rm Num}(X)_\R$
are fixed, then a constant is {\it universal} if it depends only on
$\mathfrak t(E), \omega$ and $\beta$.
\end{definition}

\begin{lemma} \label{unstable-bnd}
Let notations be from Subsect.~\ref{subsect_Large},
and let $\omega \in \text{\rm Num}(X)_\Q$.
If $E \in D^b(X)$ is not $(Z_\Omega, \mathcal P_\Omega)$-semistable,
then there exists a positive number $M$,
depending only on $\mathfrak t(E), \omega$ and $\beta$, such that
$E$ is not $(Z_m, \mathcal P_m)$-semistable for all $m \ge M$.
\end{lemma}
\noindent
{\it Proof.}
It suffices to prove the statement for $E \in \mathcal P_\Omega((0, 1])
= \mathcal A^\sharp_{(\omega, \beta \omega)}$. Let
\begin{eqnarray}   \label{AEB}
0 \to A \to E \to B \to 0
\end{eqnarray}
be an exact sequence in $\mathcal P_\Omega((0, 1])$ destablizing $E$
such that the object $B$ is $(Z_\Omega, \mathcal P_\Omega)$-semistable.
Then $\phi \big (Z_\Omega(E)(m) \big ) >
\phi \big (Z_\Omega(B)(m) \big )$ for $m \gg 0$. So (\ref{EB2})
holds for $m \gg 0$. By Lemma~\ref{En}~(i), $c_1(E) \cdot \omega -
\rk(E) \, \beta \omega \ge 0$ and $c_1(B) \cdot \omega -
\rk(B) \, \beta \omega > 0$. If $c_1(E) \cdot \omega -
\rk(E) \, \beta \omega = 0$, then $E$ is $(Z_\Omega, \mathcal P_\Omega)$-semistable
which contradicts to our assumption. So $c_1(E) \cdot \omega - \rk(E) \,
\beta \omega > 0$. Then, we have
\begin{eqnarray}   \label{EB1}
c_1(E) \cdot \omega - \rk(E) \, \beta \omega > 0, \qquad
c_1(B) \cdot \omega - \rk(B) \, \beta \omega > 0.
\end{eqnarray}
Now our proof is divided into three cases: $\rk(E) = 0$, $\rk(E) > 0$ and $\rk(E) < 0$.

\medskip
{\bf Case 1: $\rk(E) = 0$}. Then $c_1(E) \cdot \omega > 0$ by (\ref{EB1}).
Since (\ref{EB2}) holds for $m \gg 0$, we have $\rk(B) \ge 0$.
If $\rk(B) = 0$, then (\ref{EB2}) holds for all $m > 0$.
So $B$ destablizes $E$ for all $m > 0$, and we can take $M = 1$.
In the following, we assume that $\rk(B) > 0$.
By Lemma~\ref{4.2}, $B$ is a torsion free $\mu_\omega$-semistable
sheaf with $\mu_\omega(B) > \beta \omega$. From (\ref{AEB}),
we obtain an exact sequence of sheaves
$
0 \to \mathcal H^0(A) \to \mathcal H^0(E) \to B \to 0.
$
So $\rk(\mathcal H^0(E)) > 0$. Going backwards,
let $\W B$ to be the HN-filtration quotient of
$\mathcal H^0(E)$ with smallest $\mu_\omega$-slope. Then,
$\mu_\omega(\mathcal H^0(E)/{\rm Tor}(\mathcal H^0(E))) \ge
\mu_\omega(\W B)$ and $\W B$ is $\mu_\omega$-semistable.
Since $\mathcal H^0(E) \in \mathcal T_{(\omega, \beta \omega)}$,
we also have $\mu_\omega(\W B) > \beta \omega$. Therefore,
\begin{eqnarray}   \label{W-B1}
\mu_\omega(\mathcal H^0(E)/{\rm Tor}(\mathcal H^0(E))) \ge
\mu_\omega(\W B) > \beta \omega.
\end{eqnarray}
By Lemma~\ref{quotient}~(i), we have an exact sequence
$0 \to \W A \to E \to \W B \to 0$ in
$\mathcal A^\sharp_{(\omega, \beta \omega)}$
which destablizes $E$ in view of (\ref{EB2}) (replace $B$
there by $\W B$). Hence, replacing $B$ in (\ref{AEB}) by
$\W B$, we may assume in (\ref{AEB}) that $B = B[0]$ satisfies:
\begin{eqnarray}   \label{W-B2}
\mu_\omega(\mathcal H^0(E)/{\rm Tor}(\mathcal H^0(E))) \ge
\mu_\omega(B) > \beta \omega.
\end{eqnarray}
Note that $\rk(\mathcal H^{-1}(E)) = \rk(\mathcal H^0(E)) > 0$.
Since $\mathcal H^{-1}(E) \in \mathcal F_{(\omega, \beta \omega)}$,
$\mu_\omega(\mathcal H^{-1}(E)) \le \beta \omega$.
Since $c_1(\mathcal H^0(E)) = c_1(E) +
c_1(\mathcal H^{-1}(E))$ and $c_1(E) \cdot \omega > 0$, we have
\begin{eqnarray*}
    \mu_\omega(\mathcal H^0(E)/{\rm Tor}(\mathcal H^0(E)))
&\le&\mu_\omega(\mathcal H^0(E)) = {c_1(\mathcal H^0(E)) \cdot \omega \over
   \rk(\mathcal H^0(E))}  \\
&=&{(c_1(E) + c_1(\mathcal H^{-1}(E))) \cdot \omega \over
  \rk(\mathcal H^{-1}(E))} \le c_1(E) \cdot \omega + \beta \omega.
\end{eqnarray*}
Combining with (\ref{W-B2}), $(c_1(E) \cdot \omega + \beta \omega)
\ge \mu_\omega(B) > \beta \omega$.
By Lemma~\ref{compare-bnd}~(ii), $c(B)/\rk(B)$ is bounded
from below by a universal constant. By (\ref{EB2}), there exists a constant $M$,
depending only on $\mathfrak t(E), \omega$ and $\beta$, such that whenever $m \ge M$,
$
\phi \big (Z_\Omega(E)(m) \big ) > \phi \big (Z_\Omega(B)(m) \big )
$
and so $E$ is not $(Z_m, \mathcal P_m)$-semistable.

\medskip
{\bf Case 2: $\rk(E) > 0$}. Then $\mu_\omega(E) > \beta \omega$
by (\ref{EB1}), and $\rk(\mathcal H^0(E)) > 0$. Let
$\mathcal E = \mathcal H^{-1}(E)$. Assume that $\mathcal E \ne 0$.
Then $\mu_\omega(\mathcal E) \le \beta \omega$
since $\mathcal E \in \mathcal F_{(\omega, \beta \omega)}$.
As in Case~1, we can choose the object $B$ in
(\ref{AEB}) to be the HN-filtration quotient of
$\mathcal H^0(E)$ with smallest $\mu_\omega$-slope.
Then $B$ is semistable and satisfies
(\ref{W-B2}). By (\ref{W-B2}),
\begin{eqnarray}   \label{case2.1}
     \mu_\omega(B) - \mu_\omega(E)
&\le&\mu_\omega(\mathcal H^0(E)) - \mu_\omega(E) = {(c_1(E) + c_1(\mathcal E)) \omega \over
   \rk(E) + \rk(\mathcal E)} - \mu_\omega(E) \nonumber  \\
&=&{\mu_\omega(\mathcal E) - \mu_\omega(E) \over
   1 + \rk(E)/\rk(\mathcal E)}  \le {\beta \omega - \mu_\omega(E) \over
   1 + \rk(E)/\rk(\mathcal E)} \nonumber  \\
&\le&{\beta \omega - \mu_\omega(E) \over
   1 + \rk(E)} < 0.
\end{eqnarray}
So $\mu_\omega(E) > \mu_\omega(B) > \beta \omega$.
By Lemma~\ref{compare-bnd}~(ii), $c(B)/\rk(B)$ is bounded from
below by a universal constant. Now (\ref{EB2}) is equivalent to
\begin{eqnarray}  \label{case2.3}
  {\omega^2 m^2 \over 2} \big ( \mu_\omega(B) -
    \mu_\omega(E) \big )
< {c(B) \over \rk(B)} \, \big ( \mu_\omega(E) -
    \beta \omega \big ) - {c(E) \over \rk(E)} \,
    \big ( \mu_\omega(B) - \beta \omega \big ).
\end{eqnarray}
In view of the negative upper bound $(\beta \omega -
\mu_\omega(E))/(1 + \rk(E))$ for $\big ( \mu_\omega(B) -
\mu_\omega(E) \big )$ from (\ref{case2.1}),
there exists a constant $M$, depending only
on $\mathfrak t(E), \omega$ and $\beta$, such that
$\phi \big (Z_\Omega(E)(m) \big ) >
\phi \big (Z_\Omega(B)(m) \big )$ whenever $m \ge M$.
Hence our lemma holds.

Let $\mathcal E = 0$. Then $E = \mathcal H^0(E)$ has positive rank.
If ${\rm Tor}(E)$ contains a $0$-dimensional subsheaf $Q$,
then $Q \in \mathcal A^\sharp_{(\omega, \beta \omega)}$ is
a proper sub-object of $E$ destablizing $E$ with respect to
$(Z_m, \mathcal P_m)$ for all $m > 0$ and we are done.
If ${\rm Tor}(E)$ is a $1$-dimensional torsion,
then we can choose $B$ in (\ref{AEB}) to be the HN-filtration
quotient of $E$ with smallest $\mu_\omega$-slope.
Now $B$ is $\mu_\omega$-semistable and satisfies
$$
\beta \omega < \mu_\omega(B) \le \mu_\omega(E/{\rm Tor}(E)) \le
\mu_\omega(E) - 1/\rk(E).
$$
So $\mu_\omega(B) - \mu_\omega(E) \le -1/\rk(E)$.
Again $c(B)/\rk(B)$ is bounded from below by a universal constant.
In view of (\ref{case2.3}), our lemma holds. In the following,
assume that ${\rm Tor}(E) = 0$. Let $\W B$ be the HN-filtration
quotient of $E$ with smallest $\mu_\omega$-slope. Then $\W B$
is $\mu_\omega$-semistable and satisfies the inequalities
$
\mu_\omega(E) \ge \mu_\omega(\W B) > \beta \omega.
$
If $\mu_\omega(E) > \mu_\omega(\W B)$, then we can choose
the object $B$ in (\ref{AEB}) such that $B = \W B$.
Since $\rk(B) < \rk(E)$, the rational number $\mu_\omega(B) - \mu_\omega(E)$
is bounded from above by a negative universal constant.
Hence in view of (\ref{case2.3}), our lemma holds.
We are left with the case when $\mu_\omega(E) = \mu_\omega(\W B)$,
i.e., $E = \W B$ is $\mu_\omega$-semistable with $\mu_\omega(E)
> \beta \omega$. We claim that this is impossible.
Indeed, we see from (\ref{AEB}) that $A \ne 0$ is
a torsion free sheaf and sits in the exact sequence
\begin{eqnarray}   \label{calE=0.1}
0 \to \mathcal H^{-1}(B) \to A \to E \to \mathcal H^0(B) \to 0.
\end{eqnarray}
Since $\phi \big (Z_\Omega(E)(m) \big ) < \phi \big (Z_\Omega(A)(m)
\big )$ for $m \gg 0$, we see from Remark~\ref{compare-phase} that
\begin{eqnarray}   \label{AE1}
  {\omega^2 m^2 \over 2} \big ( \mu_\omega(A) -
    \mu_\omega(E) \big )
> {c(A) \over \rk(A)} \, \big ( \mu_\omega(E) -
    \beta \omega \big ) - {c(E) \over \rk(E)} \,
    \big ( \mu_\omega(A) - \beta \omega \big )
\end{eqnarray}
for $m \gg 0$. So $\mu_\omega(A) \ge \mu_\omega(E)$. If $\mu_\omega(A) = \mu_\omega(E)$,
then (\ref{AE1}) holds for all $m > 0$ and our lemma holds by taking $M = 1$.
In the following, assume that $\mu_\omega(A) > \mu_\omega(E)$.
Since $E$ is $\mu_\omega$-semistable, $\mathcal B :=
\mathcal H^{-1}(B) \ne 0$ by (\ref{calE=0.1}). By Lemma~\ref{4.2},
$\mathcal B$ is $\mu_\omega$-semistable with
$\mu_\omega(\mathcal B) \le \beta \omega$,
and $\mathcal H^0(B)$ is a $0$-dimensional torsion sheaf.
Let $\mathcal G$ be the image of the map $A \to E$ from
(\ref{calE=0.1}). Then we have two exact sequences of sheaves:
\begin{eqnarray}
&0 \to \mathcal B \to A \to \mathcal G  \to 0,&
   \label{calE=0.2}   \\
&0 \to \mathcal G  \to E \to \mathcal H^0(B) \to 0.&
   \label{calE=0.3}
\end{eqnarray}
By (\ref{calE=0.3}), $\mu_\omega(\mathcal G) = \mu_\omega(E)
< \mu_\omega(A)$. So $\mu_\omega(\mathcal B) > \mu_\omega(A)$
by (\ref{calE=0.2}). However, this contradicts to
$\mu_\omega(\mathcal B) \le \beta \omega < \mu_\omega(E) < \mu_\omega(A)$.

\medskip
{\bf Case 3: $\rk(E) < 0$}. Let $\mathcal E =
\mathcal H^{-1}(E)$. Then $\mu_\omega(E) < \beta \omega$
by (\ref{EB1}), and $\mathcal E \ne 0$ is torsion free.
Assume that $\rk(\mathcal H^0(E)) > 0$.
As in Case 1, we can choose the object $B$ in
(\ref{AEB}) to be the HN-filtration quotient of
$\mathcal H^0(E)$ with smallest $\mu_\omega$-slope.
Then $B$ is $\mu_\omega$-semistable and satisfies
(\ref{W-B2}), and $\mu_\omega(\mathcal E) \le \beta \omega$
since $\mathcal E \in \mathcal F_{(\omega, \beta \omega)}$.
By (\ref{W-B2}),
\begin{eqnarray}   \label{case3.1}
    \mu_\omega(B) - \mu_\omega(E)
&\le&\mu_\omega(\mathcal H^0(E)) - \mu_\omega(E)   \nonumber \\
&=&{\mu_\omega(\mathcal E) - \mu_\omega(E) \over \rk(\mathcal H^0(E))/\rk(\mathcal E)}
              \nonumber \\
&\le&(\beta \omega - \mu_\omega(E)) \cdot {\rk(\mathcal E) \over \rk(\mathcal H^0(E))}
              \nonumber \\
&=&(\beta \omega - \mu_\omega(E)) \cdot \left ( 1 -
   {\rk(E) \over \rk(\mathcal H^0(E))} \right )    \nonumber \\
&\le&(\beta \omega - \mu_\omega(E)) \cdot ( 1 - \rk(E)).
\end{eqnarray}
Combining with (\ref{W-B2}), we conclude that
\begin{eqnarray}   \label{case3.2}
\mu_\omega(E) + (\beta \omega - \mu_\omega(E)) \cdot ( 1 - \rk(E))
\ge \mu_\omega(B) > \beta \omega.
\end{eqnarray}
So $c(B)/\rk(B)$ is bounded from below by a constant depending
only on $\mathfrak t(E), \omega$ and $\beta$. Also, $\mu_\omega(B)
- \mu_\omega(E) > \beta \omega - \mu_\omega(E) > 0$.
Now (\ref{EB2}) is equivalent to
\begin{eqnarray}  \label{case3.3}
  {\omega^2 m^2 \over 2} \big ( \mu_\omega(B) -
    \mu_\omega(E) \big )
> {c(B) \over \rk(B)} \, \big ( \mu_\omega(E) -
    \beta \omega \big ) - {c(E) \over \rk(E)} \,
    \big ( \mu_\omega(B) - \beta \omega \big ).
\end{eqnarray}
It follows that there exists a constant $M$, depending only
on $\mathfrak t(E), \omega$ and $\beta$, such that $E$ is not
$(Z_m, \mathcal P_m)$-semistable whenever $m \ge M$.

We are left with the case $\rk(\mathcal H^0(E)) = 0$.
Assume that either $\mathcal H^{-1}(E)$ is $\mu_\omega$-unstable,
or the support of $\mathcal H^0(E)$ has dimension $1$.
Let $A$ be the HN-filtration subsheaf of
$\mathcal H^{-1}(E)$ with largest $\mu_\omega$-slope.
Then $A \in \mathcal F_{(\omega, \beta \omega)}$ is
$\mu_\omega$-semistable with $\mu_\omega(A) \le \beta \omega$.
When $\mathcal H^{-1}(E)$ is $\mu_\omega$-unstable,
$\mu_\omega(A) > \mu_\omega(\mathcal H^{-1}(E))$;
so $\mu_\omega(A) \ge \mu_\omega(\mathcal H^{-1}(E)) + d_1$ for some
positive number $d_1$ depending only on $\rk(E)$ and $\omega$.
When the support of $\mathcal H^0(E)$ has dimension $1$, we have
\begin{eqnarray*}
      \mu_\omega(A)
&\ge& \mu_\omega(\mathcal H^{-1}(E))
   = {\big ( c_1(\mathcal H^0(E)) - c_1(E) \big ) \cdot \omega
  \over \rk \, \mathcal H^{-1}(E)}        \\
&\ge& {1 - c_1(E) \cdot \omega \over - \rk(E)}
  = \mu_\omega(E) - {1 \over \rk(E)}.
\end{eqnarray*}
In either case, $\beta \omega \ge \mu_\omega(A) \ge \mu_\omega(E) + d_2$
where $d_2$ is a positive number depending only on $\rk(E)$ and
$\omega$. In particular, $\mu_\omega(A) - \mu_\omega(E) \ge d_2$.
Since $\mu_\omega(E) < \beta \omega$,
we see from Lemma~\ref{compare-bnd}~(ii) that
${c(A)/\rk(A)} \cdot \big ( \mu_\omega(E) - \beta \omega \big )$
is bounded from above by a constant depending only on
$\mathfrak t(E), \omega$ and $\beta$. Hence there exists
$M$ depending only on $\mathfrak t(E), \omega$ and $\beta$
such that (\ref{AE1}) holds whenever $m \ge M$.
By Remark~\ref{compare-phase}, $\phi \big (Z_\Omega(E)(m) \big )
< \phi \big (Z_\Omega(A[1])(m) \big )$ whenever $m \ge M$.
By Lemma~\ref{quotient}~(ii), $A[1]$ is a proper sub-object of $E$
in $\mathcal A^\sharp_{(\omega, \beta \omega)}$.
So $A[1]$ destablizes $E$ whenever $m \ge M$.

Finally, assume that $\mathcal H^{-1}(E)$ is $\mu_\omega$-semistable
and $\mathcal H^0(E)$ is a $0$-dimensional torsion sheaf.
By the exact sequence of sheaves
\begin{eqnarray}  \label{case3.4}
0 \to \mathcal H^{-1}(A) \to \mathcal E \to \mathcal B \to
\mathcal H^0(A) \to \mathcal H^0(E) \to \mathcal H^0(B) \to 0,
\end{eqnarray}
$\mathcal H^0(B)$ is a $0$-dimensional torsion sheaf. Since $B$
destablizes $E$ with respect to $(Z_\Omega, \mathcal P_\Omega)$,
$B$ can not be a $0$-dimensional torsion sheaf.
By Lemma~\ref{4.2}, $\mathcal B := \mathcal H^{-1}(B)$ is
a torsion free $\mu_\omega$-semistable sheaf with
$\mu_\omega(\mathcal B) \le \beta \omega$. Since $\mu_\omega(E)
= \mu_\omega(\mathcal E)$ and $\mu_\omega(B) =
\mu_\omega(\mathcal B)$, (\ref{EB2}) is equivalent to
\begin{eqnarray}     \label{case3.5}
  {\omega^2 m^2 \over 2} \big ( \mu_\omega(\mathcal B) -
    \mu_\omega(\mathcal E) \big )
< {c(B) \over \rk(B)} \, \big ( \mu_\omega(\mathcal E)
    - \beta \omega \big ) - {c(E) \over \rk(E)} \,
    \big ( \mu_\omega(\mathcal B) - \beta \omega \big ).
\end{eqnarray}
Since it holds for $m \gg 0$, $\mu_\omega(\mathcal B)
\le \mu_\omega(\mathcal E)$. If $\mu_\omega(\mathcal B)
= \mu_\omega(\mathcal E)$, then (\ref{case3.5}) holds for all
$m \ge 1$; so our lemma is true with $M = 1$.
Let $\mu_\omega(\mathcal B) < \mu_\omega(\mathcal E)$.
Then $\mu_\omega(\mathcal B) < \mu_\omega(\mathcal E) <
\beta \omega$. Since $\mathcal E$ and $\mathcal B$ are
$\mu_\omega$-semistable, the map $\mathcal E \to \mathcal B$ in
(\ref{case3.4}) is zero. So we obtain the exact sequence
$
0 \to \mathcal B \to \mathcal H^0(A) \to \mathcal H^0(E)
\to \mathcal H^0(B) \to 0.
$
Since $\mathcal H^0(A) \in \mathcal T_{(\omega, \beta \omega)}$,
we get the contradiction
\begin{equation}
\beta \omega < \mu_\omega(\mathcal H^0(A)) = \mu_\omega(\mathcal B)
< \beta \omega.   \tag*{$\qed$}
\end{equation}

\begin{lemma} \label{semistable-bnd}
Let notations be from Subsect.~\ref{subsect_Large},
and let $\omega \in \text{\rm Num}(X)_\Q$.
If an object $E \in D^b(X)$ is $(Z_\Omega, \mathcal P_\Omega)$-semistable,
then there exists a positive $M$,
depending only on $\mathfrak t(E), \omega$ and $\beta$, such that
$E$ is $(Z_m, \mathcal P_m)$-semistable for all $m \ge M$.
\end{lemma}
\begin{proof}
It suffices to prove the statement for $E \in \mathcal P_\Omega((0, 1])
= \mathcal A^\sharp_{(\omega, \beta \omega)}$.
We begin with an observation. Consider the set
\begin{eqnarray}   \label{BEb}
W = \{ w \in [1,+\infty) |\,  E \text{ is $(Z_w, \mathcal P_w)$-unstable}\}.
\end{eqnarray}
If $W$ is empty, then we are done by taking $M = 1$. Assume that $W$ is nonempty.
By Lemma~\ref{4.1}, $E$ is $(Z_m, \mathcal P_m)$-semistable for $m \gg 0$.
So for every $w \in W$, we can find a maximal destablizing sub-object $A_w \in
\mathcal A^\sharp_{(\omega, \beta \omega)}$ of $E$
with respect to $(Z_w, \mathcal P_w)$, satisfying the properties listed
in Lemma~\ref{bnd-rk}. By Lemma~\ref{bnd-rk}, there exists a universal constant $N$
(depending only on $\mathfrak t(E)$, $\beta$ and $\omega$) such that
\begin{eqnarray}   \label{rkAbN}
|\rk(A_w)| \le N.
\end{eqnarray}
We need to show that $W$ has a universal upper bound. To show this,
it suffices to prove that, given any exact sequence in $\As$:
\begin{eqnarray}   \label{eqn-AEB}
0 \to A_w \to E \to B_w \to 0
\end{eqnarray}
where $E$ is $(Z_w, \mathcal P_w)$-unstable for some $w \in W$ and $A_w$ is the maximal
destablizing sub-object with respect to $(Z_w, \mathcal P_w)$, we can find a constant $M>0$
depending only on $\mathfrak t(E), \omega$ and $\beta$ such that
$\phi \big (Z_\Omega(A_w)(m) \big ) \le \phi \big (Z_\Omega(E)(m) \big )$ whenever $m>M$, i.e.,
\begin{eqnarray}   \label{setup100}
&   &{\omega^2 m^2 \over 2} \big ( \rk(E) \, c_1(A_w) \omega -
         \rk(A_w) \, c_1(E) \omega \big )   \nonumber  \\
&\le&c(A_w) \, \big ( c_1(E) \omega - \rk(E) \, \beta \omega \big )
    - c(E) \, \big ( c_1(A_w) \omega - \rk(A_w) \, \beta \omega \big ).
\end{eqnarray}
whenever $m \ge M$, in view of the discussions in Remark~\ref{compare-phase}.
So fix such an exact sequence \eqref{eqn-AEB}.
Note that $E$ satisfies Lemma~\ref{4.2}~(i), (ii) or (iii). In the following,
our proof is divided into three cases accordingly.

\medskip
{\bf Case 1:} $E$ satisfies Lemma~\ref{4.2}~(i).
If $E$ is a $0$-dimensional torsion sheaf, then it is
$(Z_m, \mathcal P_m)$-semistable for all $m > 0$, contradicting to the nonemptiness of $W$.
So $E$ must be a $1$-dimensional torsion sheaf, and (\ref{setup100}) is simplified to
\begin{eqnarray}   \label{semistable-bnd.1}
{\omega^2 m^2 \over 2} \big ( -\rk(A_w) \, c_1(E) \omega \big )
\le c(A_w) \, c_1(E) \omega - c(E) \, \big ( c_1(A_w) \omega - \rk(A_w) \, \beta \omega \big )
\end{eqnarray}
Note from the long exact sequence of cohomology of (\ref{eqn-AEB}) that $A_w$ is a sheaf and
\begin{eqnarray}   \label{semistable-bnd.2}
0 \to \mathcal H^{-1}(B_w) \to A_w \to E \to \mathcal H^0(B_w) \to 0
\end{eqnarray}
is an exact sequence of sheaves.
Since (\ref{semistable-bnd.1}) holds for $m \gg 0$ but does not hold for $m = w$,
$\rk(A_w) > 0$. By \eqref{rkAbN}, $0 < \rk(A_w) \le N$.
Since $\rk \big (\mathcal H^{-1}(B_w) \big ) = \rk(A_w)$,
$0 < \rk \big (\mathcal H^{-1}(B_w) \big ) \le N$.
By the definition of $\mathcal A^\sharp_{(\omega, \beta \omega)}$,
we have $\mu_\omega (A_w) > \beta \omega$. So
\begin{eqnarray}   \label{semistable-bnd.3}
c_1(A_w) \omega > \rk(A_w) \cdot \beta \omega \ge -N \cdot |\beta \omega|.
\end{eqnarray}
Similarly, we have $\mu_\omega \big (\mathcal H^{-1}(B_w) \big ) \leq \beta \omega$.
It follows that
$$
c_1\big (\mathcal H^{-1}(B_w) \big ) \omega
\le \rk\big (\mathcal H^{-1}(B_w) \big ) \cdot \beta \omega \le N \cdot |\beta \omega|.
$$
Note that $\mathcal H^0(B_w)$ is a torsion sheaf. Thus $c_1\big (\mathcal H^0(B_w) \big ) \ge 0$ and
\begin{eqnarray}   \label{semistable-bnd.4}
c_1(B_w) \omega
= c_1\big (\mathcal H^0(B_w) \big ) \omega - c_1\big (\mathcal H^{-1}(B_w) \big ) \omega
\ge -c_1\big (\mathcal H^{-1}(B_w) \big ) \omega \ge -N \cdot |\beta \omega|.
\end{eqnarray}
Since $c_1(A_w) = c_1(E) - c_1(B_w)$, we see from \eqref{semistable-bnd.3} and
\eqref{semistable-bnd.4} that
\begin{eqnarray*}
-N \cdot |\beta \omega| \le c_1(A_w) \omega \le c_1(E) \omega + N \cdot |\beta \omega|.
\end{eqnarray*}
So $\rk(A_w)$, $|c_1(A_w) \omega|$ and $|\mu_\omega (A_w)|$ are bounded from above
by universal constants.

Consider the usual HN-filtration of the sheaf $A_w$ with respect to $\mu_\omega$:
$$
{\rm Tor}({A_w}) = A_0 \subset A_1 \subset \cdots \subset A_n = A_w
$$
where $n \le \rk(A_w) \le N$. Then $\mu_\omega (A_w) \geq \mu_\omega (A_n/A_{n-1})$.
By the definition of $\As$, we have $\mu_\omega (A_n/A_{n-1}) > \beta \omega$.
Hence $\rk(A_{n-1})$, $|c_1(A_{n-1}) \omega|$, $|\mu_\omega (A_{n-1})|$ and
$|\mu_\omega (A_n/A_{n-1})|$ are bounded from above by universal constants.
Similarly, using $A_{n-1}$ instead of $A_n = A_w$, we see that
$\rk(A_{n-2})$, $|c_1(A_{n-2}) \omega|$, $|\mu_\omega (A_{n-2})|$ and
$|\mu_\omega (A_{n-1}/A_{n-2})|$ are bounded from above by universal constants.
Repeating this process, we conclude that
$\rk(A_i), |c_1(A_i) \omega|, |\mu_\omega (A_i)|$ and $|\mu_\omega (A_i/A_{i-1})|$,
with $1 \le i \le n$, are all bounded from above by a universal constant.
Applying Lemma~\ref{compare-bnd}~(ii)
to the torsion free $\mu_\omega$-semistable sheaves $A_i/A_{i-1}$, we see that all the
numbers $c(A_i/A_{i-1})$, $1 \le i \le n$, are bounded from below by a universal constant.
Suppose ${\rm Tor}({A_w}) \neq 0$. To understand $c(A_0) = c({\rm Tor}({A_w}))$, note from
(\ref{semistable-bnd.2}) that ${\rm Tor}({A_w})$ does not contain any 0-dimensional subsheaf
because $E$ is $(Z_\Omega, \mathcal P_\Omega)$-semistable, $\rk(E) = 0$ and $c_1(E) > 0$.
So ${\rm Tor}({A_w})$ is a 1-dimensional torsion sheaf. Since $H^{-1}(B_w)$ is torsion free,
the subsheaf ${\rm Tor}({A_w})$ of $A_w$ is mapped injectively into $E$. Therefore,
$0 < c_1 \big ({\rm Tor}({A_w}) \big )\omega \le c_1(E) \omega$. Note that the sheaf injection
${\rm Tor}({A_w}) \hookrightarrow A_w$ is also an injection in $\As$. So ${\rm Tor}({A_w})$
is a sub-object of $E$ in $\As$.  By the $(Z_\Omega, \mathcal P_\Omega)$-semistability of $E$
and using (\ref{EB2}), we see that $c \big ({\rm Tor}({A_w}) \big )$ is bounded from
below by a universal constant. Overall, we have proved that
$$
c(A_0), \qquad c(A_i/A_{i-1}),
$$
with $1 \le i \le n$, are bounded from below by a universal constant.
Note from (\ref{cE}) that $c(A_w) = c(A_0) + \sum_{i=1}^n c(A_i/A_{i-1})$.
Since $n \le N$, $c(A_w)$ is bounded from below by a universal constant.
Hence $c(A_w)/\rk(A_w)$ is bounded from below by a universal constant.
By (\ref{semistable-bnd.1}), there exists a universal constant $M$
such that $\phi \big (Z_\Omega(A_w)(m) \big ) \le
\phi \big (Z_\Omega(E)(m) \big )$ whenever $m \ge M$.

\medskip
{\bf Case 2:} $E$ satisfies Lemma~\ref{4.2}~(ii). We see from the exact sequence (\ref{eqn-AEB})
that $A_w$ is a torsion free sheaf. So (\ref{setup100}) is equivalent to
\begin{eqnarray}   \label{semistable-bnd-(ii).1}
{\omega^2 m^2 \over 2} \big ( \mu_\omega(A_w)  - \mu_\omega(E) \big )
\le {c(A_w) \over \rk(A_w)} \, \big ( \mu_\omega(E)  -  \beta \omega \big )
    - {c(E) \over \rk(E)} \, \big ( \mu_\omega(A_w)  -  \beta \omega \big ).
\end{eqnarray}
Since (\ref{semistable-bnd-(ii).1}) holds for $m \gg 0$ but does not hold for $m = w$,
we must have $\mu_\omega(A_w) < \mu_\omega(E)$.
Since $A_w \in \As$, we get $\mu_\omega(A_w) > \beta \omega$.
Thus, $\beta \omega < \mu_\omega(A_w) < \mu_\omega(E)$. By (\ref{rkAbN}),
$\rk(A_w)$ is bounded from above by a universal number. Hence the negative rational number
$\big ( \mu_\omega(A_w)  - \mu_\omega(E) \big )$ has a universal negative upper bound.
Next, consider the HN-filtration of the torsion free sheaf $A_w$ with respect to $\mu_\omega$:
$$
0 = A_0 \subset A_1 \subset \cdots \subset A_n = A_w.
$$
Since $\rk(A_w)$, $|c_1(A_w) \omega|$ and $|\mu_\omega (A_w)|$ are bounded from above
by universal constants, the same argument as in the previous paragraph proves that
${c(A_w)/\rk(A_w)}$ is bounded from below by a universal constant.
By (\ref{semistable-bnd-(ii).1}), there exists a universal constant $M$
such that $\phi \big (Z_\Omega(A_w)(m) \big ) \le
\phi \big (Z_\Omega(E)(m) \big )$ whenever $m \ge M$.

\medskip
{\bf Case 3:} $E$ satisfies Lemma~\ref{4.2}~(iii). Note that (\ref{setup100}) is equivalent to
\begin{eqnarray}   \label{semistable-bnd-iii.2}
& &{\omega^2 m^2 \over 2} \big ( \rk(E) \, c_1(B_w) \omega -
    \rk(B_w) \, c_1(E) \omega \big )   \nonumber  \\
&\ge&c(B_w) \, \big ( c_1(E) \omega - \rk(E) \, \beta \omega \big )
    - c(E) \, \big ( c_1(B_w) \omega - \rk(B_w) \, \beta \omega \big ).
\end{eqnarray}
Since $\mathcal H^0(E)$ is a $0$-dimensional torsion sheaf, so is $\mathcal H^0(B_w)$.
Put $\mathcal B = \mathcal H^{-1}(B_w)$.
Since (\ref{semistable-bnd-iii.2}) holds for $m \gg 0$ but does not hold for $m = w$,
$B_w$ can not be a $0$-dimensional torsion sheaf. In particular, $B_w \ne \mathcal H^0(B_w)$.
So $\mathcal B \ne 0$. Note that $\mathcal B$ is torsion free.
Now the inequality (\ref{semistable-bnd-iii.2}) is equivalent to
\begin{eqnarray}     \label{semistable-bnd-iii.3}
  {\omega^2 m^2 \over 2} \big ( \mu_\omega(\mathcal B) -
    \mu_\omega(E) \big )
\ge {c(B_w) \over \rk(\mathcal B)} \, \big ( \beta \omega -
    \mu_\omega(E) \big ) - {c(E) \over \rk(E)} \,
    \big ( \mu_\omega(\mathcal B) - \beta \omega \big ).
\end{eqnarray}
Since (\ref{semistable-bnd-iii.3}) holds for $m \gg 0$ but does not hold for $m = w$,
$\mu_\omega(\mathcal B) > \mu_\omega(E)$.
Since $E, B_w \in \As$, we have $\mu_\omega(E),
\mu_\omega(B_w) \le \beta \omega$ by Lemma~\ref{En}~(i). Thus,
\begin{eqnarray}     \label{semistable-bnd-iii.4}
\beta \omega \ge \mu_\omega(\mathcal B) > \mu_\omega(E).
\end{eqnarray}
By (\ref{rkAbN}), $|\rk(A_w)|$ is bounded from above by a universal constant. So
$$
\rk(\mathcal B) = |\rk(B_w)| \le |\rk(E)| + |\rk(A_w)|
$$
is bounded from above by a universal constant. Thus the positive rational number
$\big ( \mu_\omega(\mathcal B) - \mu_\omega(E) \big )$ has a universal positive lower bound.
In view of (\ref{semistable-bnd-iii.3}), to prove our lemma,
it remains to show that there exists a universal constant $\W N$ such that
\begin{eqnarray}   \label{semistable-bnd-(iii).5}
{c(B_w) \over \rk(\mathcal B)} \le \W N.
\end{eqnarray}

Since $\mathcal H^0(B_w)$ is a $0$-dimensional torsion sheaf, we have
\begin{eqnarray*}
   c(B_w)
= -c(\mathcal B) - \ch_2(\mathcal H^0(B_w)) \le -c(\mathcal B)
= -\sum_{i=1}^s c(\mathcal B_i)
\end{eqnarray*}
where $\mathcal B_1, \ldots, \mathcal B_s$ are the usual
HN-filtration quotients of $\mathcal B$ with respect to $\mu_\omega$
satisfying $\mu_\omega(\mathcal B_1) > \ldots >
\mu_\omega(\mathcal B_s)$. To prove (\ref{semistable-bnd-(iii).5}),
it suffices to show that each $c(\mathcal B_i)/\rk(\mathcal B_i)$
is bounded from below by a universal constant.

Finally, we analyze $\mathcal B_i$. Since $\mathcal B =
\mathcal H^{-1}(B_w) \in \mathcal F_{(\omega, \beta \omega)}$,
we see from the definition of $\mathcal F_{(\omega, \beta \omega)}$
that $\mathcal B_i \in \mathcal F_{(\omega, \beta \omega)}$ and
$\mu_\omega(\mathcal B_i) \le \beta \omega$. Let $\mathcal E = \mathcal H^{-1}(E)$.
Let $\mathcal F$ (respectively, $\mathcal G$) be the image (respectively, cokernel)
of the map $\mathcal E \to \mathcal B$ induced from (\ref{eqn-AEB}).
Combining the map $\mathcal E \to \mathcal B$ with the surjection
$\mathcal B \to \mathcal B_s$, we obtain a map $f: \mathcal E \to \mathcal B_s$.
Let $\W{\mathcal F}$ be the image of $f$. If $\W{\mathcal F} = 0$,
then we get an induced surjection $\mathcal G \cong
\mathcal B/\mathcal F \to \mathcal B_s$.
Since $\Hom(\mathcal T_{(\omega, \beta \omega)},
\mathcal F_{(\omega, \beta \omega)}) = 0$, this is impossible by
Lemma~\ref{GCQ} (note that there exists an exact sequence of sheaves
$
0 \to \mathcal G \to \mathcal H^0(A_w) \to Q \to 0
$
where $Q$ is a subsheaf of the $0$-dimensional torsion sheaf
$\mathcal H^0(E)$). Thus, $\W{\mathcal F} \ne 0$.
Since $\mathcal E$ and $\mathcal B_s$ are $\mu_\omega$-semistable,
$\mu_\omega(\mathcal E) \le \mu_\omega(\W{\mathcal F})
\le \mu_\omega(\mathcal B_s)$. Therefore, we obtain
$\mu_\omega(E) = \mu_\omega(\mathcal E)
\le \mu_\omega(\mathcal B_i) \le \beta \omega$
for every $i = 1, \ldots, s$. By Lemma~\ref{compare-bnd}~(ii),
each $c(\mathcal B_i)/\rk(\mathcal B_i)$ is bounded from below by a universal constant.
\end{proof}

\begin{theorem} \label{mggM}
Let notations be from Subsect.~\ref{subsect_Large}, and let
$\omega \in \text{\rm Num}(X)_\Q$. Fix a type $\mathfrak t =
(r, c_1, c_2)$. Then there exists a positive number $M$,
depending only on $\mathfrak t, \omega$ and $\beta$, such that
$\overline{\mathfrak M}_{{\bf u}_m}(\mathfrak t) =
\overline{\mathfrak M}_\Omega(\mathfrak t)$ for all $m \ge M$.
\end{theorem}
\begin{proof}
Follows immediately from Lemma~\ref{unstable-bnd} and
Lemma~\ref{semistable-bnd}.
\end{proof}

\begin{theorem}       \label{thm-loc-fnt}
Let $\beta, \omega \in \text{\rm Num}(X)_\Q$ with $\omega$ being ample,
and let $\mathfrak t = (r, c_1, c_2)$.
\begin{enumerate}
\item[(i)] The set of mini-walls of type $(\mathfrak t, \beta, \omega)$
in $(0, +\infty)$ is locally finite.

\item[(ii)] There exists a positive number $\W M$, depending only on
$\mathfrak t, \omega$ and $\beta$, such that there is no
mini-wall of type $(\mathfrak t, \beta, \omega)$ in $[\W M, +\infty)$.
\end{enumerate}
\end{theorem}
\begin{proof}
Part (i) is Proposition~\ref{prop-loc-fnt}. To prove (ii),
let ${\bf u} = e^{-(\beta + i \, \omega)}$ and $\W M = 1 + M$
where $M$ is the positive number from Theorem~\ref{mggM}.
If $m_0$ is a mini-wall of type $(\mathfrak t, \beta, \omega)$ in
$I = [\W M, +\infty)$, then by definition,
$\phi \big (Z_{m_0}(A) \big ) = \phi \big (Z_{m_0}(E) \big )$
where $E \in \overline{\mathfrak M}_{{\bf u}_{m_1}}(\mathfrak t)$
for some $m_1 \in I$,
$E \not \in \overline{\mathfrak M}_{{\bf u}_{m_2}}(\mathfrak t)$
for some $m_2 \in I$, and $A$ is the leading HN-filtration component
of $E$ with respect to $(Z_{m_2}, \mathcal P_{m_2})$.
In particular, $\overline{\mathfrak M}_{{\bf u}_{m_1}}(\mathfrak t)
\ne \overline{\mathfrak M}_{{\bf u}_{m_2}}(\mathfrak t)$. This
contradicts to Theorem~\ref{mggM} since $m_1, m_2 \ge \W M > M$.
\end{proof}

\begin{corollary}   \label{m01n}
Let $\beta, \omega \in \text{\rm Num}(X)_\Q$ with $\omega$ being ample.
Fix a numerical type $\mathfrak t = (r, c_1, c_2)$ and an interval
$I = [a, +\infty)$ with $a > 0$. Then there exists a finite subset
$\{ m_0^{(1)}, \ldots, m_0^{(n)} \} \subset I$, possibly empty,
such that $\overline{\mathfrak M}_{{\bf u}_{m_1}}(\mathfrak t)
= \overline{\mathfrak M}_{{\bf u}_{m_2}}(\mathfrak t)$ whenever
$m_1$ and $m_2$ are contained in the same connected component of
$I - \{ m_0^{(1)}, \ldots, m_0^{(n)} \}.$
\end{corollary}
\begin{proof}
Let $\W M$ be from Theorem~\ref{thm-loc-fnt}~(ii). If $\W M \le a$,
then the result is true by Theorem~\ref{thm-loc-fnt}~(ii) and
Lemma~\ref{chamber-moduli}. If $\W M > a$, then let
$\{ m_0^{(1)}, \ldots, m_0^{(n-1)} \}$ be the finite set of
mini-walls of type $(\mathfrak t, \beta, \omega)$ in $[a, \W M]$.
Letting $m_0^{(n)} = \W M$, we are done.
\end{proof}

\section{\bf Identify $\overline{\mathfrak M}_\Omega(\mathfrak t)$
with Gieseker and Uhlenbeck moduli spaces}
\label{sect_Gieseker}

Fix a numerical type $\mathfrak t = (r, c_1, c_2)$. We want to
compare the spaces $\overline{\mathfrak M}_\Omega(\mathfrak t)$
with the Gieseker/Simpson and Uhlenbeck spaces where $\Omega$
comes from Subsect.~\ref{subsect_Large}.
In view of Lemma~\ref{00n} and Lemma~\ref{0c1c2}, we will assume
that $r \ne 0$. The results here are similar to
those in Sect.~5 of \cite{LQ} which only considered objects $E \in
\mathcal A^p$ for those stability data $\Omega = (\omega, \rho, p, U)$
such that $\rho = (\rho_0, \rho_1, \rho_2)$ satisfies $\phi(\rho_0)
\ne \phi(-\rho_2)$. In the present case, we have $\phi(\rho_0)
= \phi(-\rho_2)$ since $\rho_0 = -1$ and $\rho_2 = 1/2$. Moreover,
we will study objects $E \in \mathcal P_\Omega((0, 1]) =
\mathcal A^\sharp_{(\omega, \beta\omega)}$ instead of objects
$E \in \mathcal A^p$ by noticing that the abelian categories
$\mathcal A^\sharp_{(\omega, \beta\omega)}$ and $\mathcal A^p$ are different.

\begin{lemma} \label{>}
Let $\Omega = (\omega, \rho, p, U)$ be from
Subsect.~\ref{subsect_Large}.
Fix a numerical type $\mathfrak t = (r, c_1, c_2)$ with $r > 0$
and $c_1 \omega/r > \beta \omega$.
Assume that $\omega$ lies in a chamber of type $\mathfrak t$.
Then, every object in $\overline{\mathfrak M}_\Omega(\mathfrak t)$
is $(Z_\Omega, \mathcal P_\Omega)$-stable, and
$\overline{\mathfrak M}_\Omega(\mathfrak t)
= \overline{\mathfrak M}_{\omega}(\mathfrak t)$.
\end{lemma}
\begin{proof}
Let $E \in \overline{\mathfrak M}_\Omega(\mathfrak t)$.
By Lemma~\ref{4.2}, $E$ is a $\mu_\omega$-semistable sheaf.
Since $\omega$ lies in a chamber of type $\mathfrak t$,
$E$ must be $\mu_\omega$-stable. In particular,
$E \in \overline{\mathfrak M}_{\omega}(\mathfrak t)$.

Conversely, let $E \in \overline{\mathfrak M}_{\omega}(\mathfrak t)$.
Then $E$ is $\mu_\omega$-stable since $\omega$ lies in a chamber of
type $\mathfrak t$. Since $\mu_\omega(E) = c_1 \omega/r >
\beta \omega$, $E \in \mathcal P_\Omega((0, 1]) =
\mathcal A^\sharp_{(\omega, \beta\omega)}$. Let $A$ be any proper
sub-object of $E$ in $\mathcal P_\Omega((0, 1])$, and let $B = E/A$.
Then we have the exact sequence $0 \to A \to E \to B \to 0$ in
$\mathcal P_\Omega((0, 1])$. So $A$ is a sheaf in
$\mathcal T_{(\omega, \beta \omega)}$ and sits in
$$
0 \to \mathcal H^{-1}(B) \to A \to E \to \mathcal H^0(B) \to 0.
$$
It follows that $A$ is torsion free with $\mu_\omega(A) >
\beta \omega$. If $\mathcal H^{-1}(B) \ne 0$, then
$\mathcal H^{-1}(B) \in \mathcal F_{(\omega, \beta \omega)}$.
So $\mu_\omega(\mathcal H^{-1}(B)) \le \beta \omega < \mu_\omega(A)$.
Thus the image $\mathcal G$ of $A \to E$ is not zero.
Since $E$ is $\mu_\omega$-stable, we conclude that
$\mu_\omega(A) < \mu_\omega(\mathcal G) \le \mu_\omega(E)$.
By (\ref{EB2}), $\phi \big (Z_\Omega(E)(m) \big ) >
\phi \big (Z_\Omega(A)(m) \big )$ for all $m \gg 0$.
If $\mathcal H^{-1}(B) = 0$, then we have an exact sequence
$0 \to A \to E \to B \to 0$ of sheaves. Since $A$ is a proper
subsheaf of $E$, $\mu_\omega(A) < \mu_\omega(E)$.
So again $\phi \big (Z_\Omega(E)(m) \big ) >
\phi \big (Z_\Omega(A)(m) \big )$ for all $m \gg 0$.
Therefore, $E$ is $(Z_\Omega, \mathcal P_\Omega)$-stable.
In particular, $E \in \overline{\mathfrak M}_\Omega(\mathfrak t)$.
\end{proof}

\begin{lemma} \label{<}
Let $\Omega = (\omega, \rho, p, U)$ be from
Subsect.~\ref{subsect_Large}.
Fix a numerical type $\mathfrak t = (r, c_1, c_2)$ with $r < 0$
and $c_1 \omega/r < \beta \omega$.
Assume that $\omega$ lies in a chamber of type $\mathfrak t$.
Let $\w {\mathfrak t} = (-r, c_1, c_1^2-c_2)$.
Then, every object in $\overline{\mathfrak M}_\Omega(\mathfrak t)$
is $(Z_\Omega, \mathcal P_\Omega)$-stable,
and $E \in \overline{\mathfrak M}_\Omega(\mathfrak t)$
if and only if $E = (\W E)^v[1]$ for some $\W E \in
\overline{\mathfrak M}_{\omega}(\w {\mathfrak t})$.
\end{lemma}
\begin{proof}
Let $E \in \overline{\mathfrak M}_\Omega(\mathfrak t)$.
Then, $(c_1(E) \cdot \omega - \rk(E) \, \beta \omega)
= c_1 \omega - r \beta \omega > 0$. By (\ref{ZEm}),
$\phi \big (Z_\Omega(E)(m) \big ) < 1$ for all $m > 0$.
So $E$ does not have any sub-objects in $\mathcal P_\Omega((0, 1])$ which are
$0$-dimensional torsion sheaves. By Lemma~\ref{4.2},
$\mathcal H^{-1}(E)$ is a torsion free $\mu_\omega$-stable sheaf
and $\mathcal H^0(E)$ is a $0$-dimensional torsion sheaf.
Note that $\mathcal H^{-1}(E)$ must be locally free (otherwise,
the $0$-dimensional torsion sheaf
$\big (\mathcal H^{-1}(E) \big)^{**}/\mathcal H^{-1}(E)$ would
be a sub-object of $E$ in $\mathcal P_\Omega((0, 1])$).
By the Lemma~3.4 in \cite{ABL}, $E = (\W E)^v[1]$ for some torsion
free sheaf $\W E$. A direct computation shows that
$\mathfrak t(\W E) = \w {\mathfrak t}$. Since $(\W E)^* =
\mathcal H^{-1}(E)$ is $\mu_\omega$-stable, so is $\W E$.
In particular, $\W E \in
\overline{\mathfrak M}_{\omega}(\w {\mathfrak t})$.

Conversely, let $E = (\W E)^v[1]$ for some $\W E \in
\overline{\mathfrak M}_{\omega}(\w {\mathfrak t})$.
Then $\mathcal H^0(E) = \mathcal Ext^1(\W E, \mathcal O_X)$ is a $0$-dimensional
torsion sheaf, and $\mathcal H^{-1}(E) = (\W E)^*$ is locally free
and $\mu_\omega$-stable with $\mu_\omega \big ( (\W E)^*\big )
= (-c_1) \omega/(-r) < \beta \omega$.
So $E \in \mathcal P_\Omega((0, 1])$. Let $A$ be any proper
sub-object of $E$ in $\mathcal P_\Omega((0, 1])$, and let $B = E/A$.
Then we have the exact sequence $0 \to A \to E \to B \to 0$
in $\mathcal P_\Omega((0, 1])$ and an exact sequence of sheaves
\begin{eqnarray}   \label{<.1}
0 \to \mathcal H^{-1}(A) \to (\W E)^* \to \mathcal H^{-1}(B) \to
\mathcal H^0(A) \to \mathcal H^0(E) \to \mathcal H^0(B) \to 0.
\end{eqnarray}
Let $\mathcal F$ and $\mathcal G$ be the image and cokernel of
$(\W E)^* \to \mathcal H^{-1}(B)$ respectively.

We claim that $A$ does not have any sub-object $Q$ in
$\mathcal P_\Omega((0, 1])$ which is a $0$-dimensional torsion sheaf.
Indeed, if such a $Q$ exists, then $Q$ is a sub-object of
$E = (\W E)^v[1]$ in $\mathcal P_\Omega((0, 1])$. In particular,
there exists a point $x \in X$ such that $\mathcal O_x$ is a sub-object of
$E = (\W E)^v[1]$ in $\mathcal P_\Omega((0, 1])$.
This leads to a contradiction:
\begin{eqnarray}     \label{contradiction}
     0
&\ne&\Hom_{\mathcal P_\Omega((0, 1])}(\mathcal O_x, (\W E)^v[1])
     \cong \Hom_{\mathcal D^b(X)}(\W E[-1], \mathcal O_x^v)   \nonumber   \\
&=&\Hom_{\mathcal D^b(X)}(\W E[-1], \mathcal O_x[-2])
     \cong \Ext_{{\rm Coh}(X)}^{-1}(\W E, \mathcal O_x) = 0
\end{eqnarray}
where we have used the fact that $\mathcal O_x^v$, the derived dual of $\mathcal O_x$,
is equal to $\mathcal O_x[-2]$.

If $\mathcal H^{-1}(A) = 0$, then $A$ is a sheaf in
$\mathcal T_{(\omega, \beta \omega)}$. If $A$ is a $1$-dimensional
torsion sheaf, then we see from (\ref{EB2}) that
$\phi \big (Z_\Omega(E)(m) \big ) >
\phi \big (Z_\Omega(A)(m) \big )$ for all $m \gg 0$. Assume that
$A$ is not a $1$-dimensional torsion sheaf. By the preceding
paragraph, $A$ can not be a $0$-dimensional torsion sheaf.
So $\rk(A) > 0$ and $\mu_\omega(A) > \beta \omega > \mu_\omega
\big ( (\W E)^*\big ) = \mu_\omega(E)$.
By (\ref{EB2}), $\phi \big (Z_\Omega(E)(m) \big ) >
\phi \big (Z_\Omega(A)(m) \big )$ for all $m \gg 0$.

If $\mathcal B := \mathcal H^{-1}(B) = 0$, then $B$ is
a $0$-dimensional torsion sheaf and
$\phi \big (Z_\Omega(B)(m) \big ) = 1$ for all $m > 0$.
By (\ref{ZEm}), $\phi \big (Z_\Omega(E)(m) \big ) <
\phi \big (Z_\Omega(B)(m) \big )$ for all $m > 0$.

In the following, assume that $\mathcal H^{-1}(A) \ne 0$ and
$\mathcal B \ne 0$. Then $\mathcal B \in
\mathcal F_{(\omega, \beta \omega)}$ is torsion free with
$\mu_\omega(\mathcal B) \le \beta \omega$.
If $\mathcal F = 0$, then $\beta \omega \ge \mu_\omega(\mathcal B)
= \mu_\omega \big ( \mathcal H^0(A) \big )$ since
$\mathcal H^0(A)/\mathcal B$ is a subsheaf of the $0$-dimensional
torsion sheaf $\mathcal H^0(E)$. This contradicts to
$\mathcal H^0(A) \in \mathcal T_{(\omega, \beta \omega)}$.
Assume that $\mathcal F \ne 0$. Then $\mathcal F$ is a proper
quotient of $(\W E)^*$. Since $(\W E)^*$ is $\mu_\omega$-stable,
$\mu_\omega\big ( (\W E)^*\big ) < \mu_\omega(\mathcal F)$.
If $\rk(\mathcal G) = 0$, then we see from the exact sequence
$0 \to \mathcal F \to \mathcal B \to \mathcal G \to 0$ that
$\mu_\omega(\mathcal B) \ge \mu_\omega(\mathcal F)
> \mu_\omega\big ( (\W E)^*\big )$; if $\rk(\mathcal G) > 0$,
then $\mu_\omega(\mathcal G) =
\mu_\omega\big ( \mathcal H^0(A) \big ) > \beta \omega$ since
$\mathcal H^0(A)/\mathcal G$ is a subsheaf of $\mathcal H^0(E)$.
Since $\mu_\omega(\mathcal B) \le \beta \omega$, we have
$\mu_\omega(\mathcal G) > \mu_\omega(\mathcal B) >
\mu_\omega(\mathcal F) > \mu_\omega\big ( (\W E)^*\big )$.
In either case, $\mu_\omega(\mathcal B) > \mu_\omega\big (
(\W E)^*\big )$. Hence $\mu_\omega(B) > \mu_\omega(E)$.
By (\ref{EB2}), $\phi \big (Z_\Omega(E)(m) \big )
< \phi \big (Z_\Omega(B)(m) \big )$ for all $m \gg 0$.
This proves that $E$ is $(Z_\Omega, \mathcal P_\Omega)$-stable.
\end{proof}

\begin{lemma} \label{=}
Let $\Omega = (\omega, \rho, p, U)$ be from
Subsect.~\ref{subsect_Large}.
Fix a numerical type $\mathfrak t = (r, c_1, c_2)$ with $r < 0$
and $c_1 \omega/r = \beta \omega$.
Let $\omega$ lie in a chamber of type $\mathfrak t$.
\begin{enumerate}
\item[(i)] If $\W E \in {\mathfrak M}_{\omega}(-r, c_1,
c_1^2-(c_2+i))$ for some $i \ge 0$ and $Q$ is a length-$i$
$0$-dimensional torsion sheaf, then $(\W E)^*[1] \oplus Q \in
\overline{\mathfrak M}_\Omega(\mathfrak t)$.

\item[(ii)] If $E \in \overline{\mathfrak M}_\Omega(\mathfrak t)$,
then $E$ is $S$-equivalent to $(\W E)^*[1] \oplus Q$ where
$Q$ is a length-$i$ $0$-dimensional torsion sheaf and
$\W E \in {\mathfrak M}_{\omega}(-r, c_1, c_1^2-(c_2+i))$.
\end{enumerate}
\end{lemma}
\begin{proof}
(i) Recall from Definition~\ref{Gr}~(iv) that $\W E$ is locally
free. Note that the numerical type of $(\W E)^*[1] \oplus Q$ is
$\mathfrak t$, and $(\W E)^*[1], Q \in \mathcal P_\Omega((0, 1])
= \mathcal A^\sharp_{(\omega, \beta\omega)}$ with
$$
\phi \big (Z_\Omega((\W E)^*[1])(m) \big )
= \phi \big (Z_\Omega(Q)(m) \big ) = 1
$$
for all $m > 0$.
Also, $Q$ is $(Z_\Omega, \mathcal P_\Omega)$-semistable.
A slight modification of the proof of Lemma~\ref{<} shows that
$(\W E)^*[1]$ is $(Z_\Omega, \mathcal P_\Omega)$-stable as well.
It follows that $(\W E)^*[1] \oplus Q$ is $(Z_\Omega,
\mathcal P_\Omega)$-semistable. Therefore, we have $(\W E)^*[1]
\oplus Q \in \overline{\mathfrak M}_\Omega(\mathfrak t)$.

(ii) Let $E \in \overline{\mathfrak M}_\Omega(\mathfrak t)$.
Since $c_1 \omega/r = \beta \omega$,
$\phi \big (Z_\Omega(E)(m) \big ) = 1$ for all $m > 0$.
By Lemma~\ref{4.2}, $\mathcal H^0(E)$ is a $0$-dimensional
torsion sheaf, and $\mathcal H^{-1}(E)$ is a torsion free
$\mu_\omega$-stable sheaf. From the exact sequence
$0 \to \mathcal H^{-1}(E)[1] \to E \to \mathcal H^0(E) \to 0$
in $\mathcal P_\Omega((0, 1])$, we see that $E$ is $S$-equivalent
to $\mathcal H^{-1}(E)[1] \oplus \mathcal H^0(E)$. Thus to prove
our result, we may assume that $E = A[1]$ for some torsion free
$\mu_\omega$-stable sheaf $A$ with $\mu_\omega(A) = \beta \omega$.
We have the canonical exact sequence $0 \to A \to A^{**} \to Q
\to 0$ where $Q$ is a $0$-dimensional torsion sheaf. It gives
rise to an exact sequence
$$
0 \to Q \to A[1] \to A^{**}[1] \to 0
$$
in $\mathcal P_\Omega((0, 1])$. Hence $E = A[1]$ is $S$-equivalent
to $A^{**}[1] \oplus Q$.
\end{proof}

\begin{theorem} \label{poly-gie-uhl}
Let $\Omega = (\omega, \rho, p, U)$ be from
Subsect.~\ref{subsect_Large}.
Fix a numerical type $\mathfrak t = (r, c_1, c_2)$.
Let $\w {\mathfrak t} = (-r, c_1, c_1^2-c_2)$.
Assume that $\omega$ lies in a chamber of type $\mathfrak t$.
\begin{enumerate}
\item[(i)] If $r > 0$, then
$\overline{\mathfrak M}_\Omega(\mathfrak t) \cong
 \overline{\mathfrak M}_{\omega}(\mathfrak t)$.

\item[(ii)] If $r < 0$ and $c_1 \omega/r < \beta \omega$,
then $\overline{\mathfrak M}_\Omega(\mathfrak t) \cong
\overline{\mathfrak M}_{\omega}(\w {\mathfrak t})$.

\item[(iii)] If $r < 0$ and $c_1 \omega/r = \beta \omega$,
then $\overline{\mathfrak M}_\Omega(\mathfrak t) \cong
\overline{\mathfrak U}_{\omega}(\w {\mathfrak t})$.
\end{enumerate}
\end{theorem}
\begin{proof}
Follows from Lemma \ref{>}, Lemma \ref{<} and Lemma \ref{=}.
Note that in (iii),
if $E \in \overline{\mathfrak M}_\Omega(\mathfrak t)$
is $S$-equivalent to $(\W E)^*[1] \oplus Q$ where
$\W E \in {\mathfrak M}_{\omega}(-r, c_1, c_1^2-(c_2+i))$ for some
$i \ge 0$ and $Q$ is a length-$i$ $0$-dimensional torsion sheaf,
then we map $E$ to
$$
\left (\W E, \quad \sum_{x \in X} h^0(X, Q_x) \cdot x \right )
$$
in ${\mathfrak M}_{\omega}(-r, c_1, c_1^2-(c_2+i)) \times
\text{\rm Sym}^i(X) \subset \overline{\mathfrak U}_{\omega}(\w
{\mathfrak t})$ (see (\ref{uhlenbeck})).
\end{proof}

\end{document}